\documentclass[11pt, reqno]{amsart}
\usepackage{amssymb}
\usepackage{graphicx}
\usepackage{amssymb}
 \usepackage{amsthm}
\usepackage{epsf}
\usepackage{amsbsy,amsmath}
\usepackage{mathtools}
\usepackage{mathrsfs}
\usepackage{amsfonts}
\usepackage{amssymb}
\usepackage{enumitem}
\usepackage{eucal}
\usepackage{graphics,mathrsfs}
\usepackage{amsthm}
\usepackage{secdot}
\usepackage{esint}
\usepackage{varwidth}
\usepackage{tasks}
\usepackage{cite}
\usepackage{mathtools,bm}
\usepackage{cmap}
\usepackage{bigints}
\usepackage{esint}
\theoremstyle{plain}
\newtheorem{theorem}{Theorem}

\newtheorem{lemma}{Lemma}

\theoremstyle{definition}

\newtheorem{remark}{Remark}

\allowdisplaybreaks
\newcommand{\N}{\mathbb N}

\newcommand{\X}{X_0^{s,p}(\Omega)}

\newcommand{\G}{\mathbb G}
\newcommand{\Om}{\Omega}
\usepackage{xcolor}

\long\def\symbolfootnote[#1]#2{\begingroup
\def\thefootnote{\fnsymbol{footnote}}\footnote[#1]{#2}\endgroup}
\usepackage{xcolor}
\usepackage[colorlinks,citecolor=blue,pagebackref=false,hypertexnames=false]{hyperref}

\numberwithin{equation}{section}

\numberwithin{equation}{section}

\begin{document}
\title[Subelliptic Critical Equations: Bifurcation, Multiplicity and Eigenvalues] {Critical Equations Involving Nonlocal Subelliptic Operators on Stratified Lie Groups: Spectrum, Bifurcation and Multiplicity}
\author{Sekhar Ghosh and Vishvesh Kumar }
\address[ Sekhar Ghosh]{Department of Mathematics, National Institute of Technology Calicut, Kozhikode, Kerala, India - 673601}
\email{sekharghosh1234@gmail.com / sekharghosh@nitc.ac.in}
\address[Vishvesh Kumar]{Department of Mathematics: Analysis, Logic and Discrete Mathematics, Ghent University, Ghent-9000, Belgium}
\email{vishveshmishra@gmail.com / vishvesh.kumar@ugent.be}

\thanks{{\em 2020 Mathematics Subject Classification: } 35R03, 35H20, 22E30, 35J20, 35R11}

\keywords{Fractional sub-Laplacian, Stratified Lie groups, Variational methods, Brezis-Nirenberg problems, Bifurcation phenomena, Critical nonlinearity, Eigenvalue problem, the Heisenberg group}

 \vskip -1cm  \hrule \vskip 1cm \vspace{-8pt}

\begin{abstract} 
In this paper, we explore the bifurcation phenomena and establish the existence of multiple solutions for the nonlocal subelliptic Brezis-Nirenberg problem:
 \begin{equation*} 
      \begin{cases} 
(-\Delta_{\mathbb{G}})^s u=  |u|^{2_s^*-2}u+\lambda u \quad &\text{in}\quad \Omega, \\
u=0\quad & \text{in}\quad \mathbb{G}\backslash \Omega,
\end{cases}
  \end{equation*}
where $(-\Delta_{\mathbb{G}})^s$ is the fractional sub-Laplacian on the stratified Lie group $\mathbb{G}$ with homogeneous dimension $Q,$ $\Omega$ is a open bounded subset of $\mathbb{G},$ $s \in (0,1)$,  $Q> 2s,$ $2_s^*:=\frac{2Q}{Q-2s}$ is subelliptic fractional Sobolev critical exponent, $\lambda>0$ is a real parameter. This work extends the seminal contributions of Cerami, Fortunato, and Struwe to nonlocal subelliptic operators on stratified Lie groups. A key component of our study involves analyzing the subelliptic $(s, p)$-eigenvalue problem for the (nonlinear) fractional $p$-sub-Laplacian $(-\Delta_{p,{\mathbb{G}}})^s$
\begin{align*} 
	(-\Delta_{p,{\mathbb{G}}})^s u&=\lambda |u|^{p-2}u,~\text{in}~\Omega,\nonumber\\
	u&=0~\text{ in }~{\mathbb{G}}\setminus\Omega,
\end{align*}
with $0<s<1<p<\infty$ and $Q>ps$,  over the  fractional Folland-Stein-Sobolev spaces on stratified Lie groups applying variational methods. Particularly, we prove that the $(s, p)$-spectrum of $(-\Delta_{p,{\mathbb{G}}})^s$ is closed and the second eigenvalue $\lambda_2(\Omega)$ with $\lambda_2(\Omega)>\lambda_1(\Omega)$ is well-defined and provides a variational characterization of $\lambda_2(\Omega)$. We emphasize that the results obtained here are also novel for $\mathbb{G}$ being the Heisenberg group.

\end{abstract}

 \maketitle 
{ \textwidth=4cm \hrule}
\tableofcontents

\section{Introduction}
\setcounter{theorem}{0}\setcounter{lemma}{0}\setcounter{definition}{0}\setcounter{proposition}{0}\setcounter{remark}{0}

	  The main aim of this paper is to study the following nonlocal subelliptic equation involving critical nonlinearities, known as the nonlocal subelliptic Brezis-Nirenberg problem: 
  \begin{equation} \label{pro1intro}
      \begin{cases} 
(-\Delta_{\mathbb{G}})^s u=  |u|^{2_s^*-2}u+\lambda u \quad &\text{in}\quad \Omega, \\
u=0\quad & \text{in}\quad \mathbb{G}\backslash \Omega,
\end{cases}
  \end{equation}
where $(-\Delta_{\mathbb{G}})^s$ is the fractional sub-Laplacian on the stratified Lie group $\mathbb{G}$ with homogeneous dimension $Q,$ $\Omega$ is a open bounded subset of $\mathbb{G},$ $s \in (0,1)$,  $Q> 2s,$ $2_s^*:=\frac{2Q}{Q-2s}$ is subelliptic fractional Sobolev critical exponent, $\lambda>0$ is a real parameter.  The operator $(-\Delta_{\mathbb{G}})^s$ is defined as 
    \begin{equation}
        (-\Delta_{\mathbb{G}})^su(x):= C(Q,s)\,\,  P.V. \int_{{\mathbb{G}} } \frac{(u(x)-u(y))}{\left|y^{-1} x\right|^{Q+2 s}} d y, \quad x \in {\mathbb{G}},
    \end{equation}
    with $|\cdot|$ being a homogenous norm on the stratified Lie group $\G,$ $C(Q,s)$ is the positive normalization constant depending only on the homogeneous dimension $Q$ of $\G$ and  $s \in (0, 1)$. Here, the symbol $P. V.$ stands for the Cauchy principal value. 
Recently, there has been considerable interest in nonlocal subelliptic operators, such as the fractional sub-Laplacian, on the Heisenberg group and more broadly on stratified Lie groups. These operators are notable for their fascinating theoretical structures and practical applications. Although it is impossible to provide an exhaustive list of references, we recommend \cite{ FF2015, PP22, GKR22, GKR24, KD20, GKR23,  GT21, GLV, K20,   WN19, RT16, FZ24,  RT20} and the references cited therein.
The fractional sub-Laplacian $(-\Delta_\G)^s$ has been extensively studied and found to be connected with various areas of mathematics. This operator, initially defined for $u \in C_c^\infty(\G)$, has the following integral representation:
\begin{equation}\label{fracl}
(-\Delta_\G)^su(x):= \lim_{\epsilon \rightarrow 0} \int_{\G \backslash B(x, \epsilon)} \frac{|u(x)-u(y)|}{|y^{-1}x|^{Q+2s}} d y = C(Q, s)\,\, P.V. \int_{\G} \frac{|u(x)-u(y)|}{|y^{-1}x|^{Q+2s}} d y.
\end{equation}
It is known that for an $H$-type group, the operator $(-\Delta_\G)^s,$ $s \in (0, \frac{1}{2}),$ is a multiple of the pseudo-differential operator defined as:
\begin{equation}
\mathcal{L}_s:=2^s (-\Delta_z)^{\frac{s}{2}} \frac{\Gamma (-\frac{1}{2}\Delta_\G (-\Delta_z)^{-\frac{1}{2}}+\frac{1+s}{2}) }{\Gamma (-\frac{1}{2}\Delta_\G (-\Delta_z)^{-\frac{1}{2}}+\frac{1-s}{2})},
\end{equation}
where $-\Delta_z$ is the positive Laplacian on the center of the $H$-type group $\G$, and $-\Delta_\G$ denotes the sub-Laplacian on $\G$. For more details, refer to \cite{BF13, RT16, RT20}. It is worth noting that $\mathcal{L}_s$ is a ``conformal invariant" operator with significant applications in CR geometry (see \cite{FMMT15}).  Moreover, we have $$\lim_{s \rightarrow 1^-} (-\Delta_\G)^s u = -\Delta_\G u $$ for all $u \in C_c^\infty(\G)$ (see \cite[Proposition 1]{PP22}). In a recent study, Garofalo and Tralli \cite{GT21} explicitly computed the fundamental solutions of the nonlocal operators $(-\Delta_\G)^s$ and $(- \Delta_{\G}^su)$ and established intertwining formulas on $H$-type groups.

 Brezis and Nirenberg in their seminal paper \cite{BN83} considered  the following nonlinear elliptic partial differential critical equation:
\begin{equation} \label{EucBN}
    \begin{cases}
        &-\Delta u = |u|^{2^*-2} u+\lambda u \quad \text{in} \,\, \Omega,\\&
        u=0 \quad \text{on} \,\, \partial \Omega
    \end{cases} 
\end{equation} for $n\geq 3$ and $2^*:= \frac{2n}{n-2},$ where $\Omega$ is a smooth bounded domain of $\mathbb{R}^n$ and $\lambda$ is a real parameter. They established the existence of a positive solution for $n \geq 3$ and $\lambda \in (0, \lambda_1)$, where $\lambda_1$ denotes the first eigenvalue of $`-\Delta'$ in $H^1_0(\Omega)$. Later, Capozzi et al. \cite{CFP85} established that for $n \geq 4$, equation \eqref{EucBN} admits a nontrivial solution for every parameter $\lambda$. Following \cite{BN83}, similar problems have been extensively studied for different classes of (pseudo)-differential operators due to their connection with issues in differential geometry and physics, where a lack of compactness occurs (see, for instance, the famous Yamabe problem). We refer   \cite{BS15, MRS16,  MM17, SER13, SER14, SERV05, SERV13} and the references therein.

The first multiplicity result of the Brezis-Nirenberg problem was established by Cerami, Fortunato, and Struwe \cite{CFS84}. They showed that in an appropriate left neighborhood of any eigenvalue of $-\Delta$ (with Dirichlet boundary data), the number of solutions is at least twice the multiplicity of the eigenvalue. The authors also provided an estimate of the length of this neighborhood, which depends on the best critical Sobolev constant, the Lebesgue measure of the domain where the problem is defined, and the space dimension. In 2015, Fiscella, Molica Bisci and Servadei \cite{FMS16} extended the result of \cite{CFS84} in the nonlocal setting. They studied the following problem: 
\begin{equation} \label{EucFBN}
    \begin{cases}
        &(-\Delta)^s u = |u|^{2_s^*-2} u+\lambda u \quad \text{in} \,\, \Omega,\\&
        u=0 \quad \text{on} \,\, \mathbb{R}^n \backslash \Omega,
    \end{cases} 
\end{equation} where $s \in (0, 1),$ $\Omega$ is an open subset on $\mathbb{R}^n, n>2s,$ $2_s^*:=2n/ (n-2s)$ is the fractional critical Sobolev exponent and $\lambda$ is a real parameter. Assume that $\lambda_k< \lambda <\lambda_{k+1},$ $$\lambda>\lambda_{k+1}-S_{2_s^*, \Omega}|\Omega|^{-\frac{2s}{n}}$$ and $m$ denotes the multiplicity of $\lambda_{k+1},$ where $0<\lambda_1<\lambda_2<\cdots \rightarrow \infty$ are the eigenvalues (repeated with multiplicity) of $(-\Delta)^s$. Then, problem \eqref{EucFBN} has at least $m$ pairs of nontrivial solutions $\pm u_{\lambda, i}$ such that $\|u_{\lambda, i}\|_{X_0^{s,2}(\Omega)} \rightarrow 0$ as $\lambda$ approaches $\lambda_{k+1}$ for any $i=1,2,\ldots, m.$ There have been several studies about the bifurcation and multiplicity results, we refer to \cite{FMS18,PSY16, MRS16} and reference therein for a complete overview. 

In this paper, we are interested in the Brezis-Nirenberg problem involving the fractional sub-Laplacian on stratified Lie groups. The motivation for considering the Brezis-Nirenberg type problem on stratified Lie groups, in particular, on the Heisenberg group, coming from it application to CR geometry.  The {\it CR Yamabe problem} was first explored in the seminal work by Jerison and Lee \cite{JL}.  The CR Yamabe problem is: 

{\it Given a compact, strictly pseudoconvex CR manifold, find a choice of contact form for which
the pseudohermitian scalar curvature is constant.} 

The Heisenberg group, the simplest type of stratified Lie group, is to CR geometry what Euclidean space is to conformal geometry. This similarity has made the analysis on stratified Lie groups essential for solving the CR Yamabe problem, which has led to significant interest in studying subelliptic PDEs on these groups. This research is also connected to finding the best constants for the Sobolev inequality on the Heisenberg group, initially studied by Folland and Stein \cite{FS82} and later by Jerison and Lee \cite{JL89}. For more on the best constants of functional inequalities on stratified Lie groups and related subelliptic variational problems, see \cite{GKR23} and its references. 
 On the other hand, the Heisenberg group appears in various fields, such as quantum theory (notably in contexts like the uncertainty principle and commutation relations) \cite{Car66, Div97}, signal theory \cite{Sch86}, the theory of theta functions \cite{Car66, Zel97}, and even number theory. Additionally, the seminal work of Rothschild and Stein \cite{RS76} highlighted the vital role of nilpotent Lie groups in providing precise subelliptic estimates for differential operators on manifolds. Their lifting theorem, known as the Rothschild-Stein theorem, demonstrates that a broad range of differential operators, particularly Hörmander's sums of squares of vector fields on manifolds, can be effectively approximated by a sub-Laplacian defined on certain stratified Lie groups.

For the sub-Laplacian on stratified Lie group, the Brezis-Nirenberg problem 
\begin{equation} \label{BNLocal}
    \begin{cases}
        -\Delta_{\G} u=|u|^{2^*-2} u+\lambda u \quad &\text{in} \quad \Omega, \\ u=0 \quad &\text{in} \quad \partial\Omega,
         \end{cases}
\end{equation}
where  $2^*=2Q/(Q-2)$ is the Folland-Sobolev critical exponent, was studied by  Loiudice \cite{L07} extending the Euclidean results in \cite{BN83}. She proved that \eqref{BNLocal} has at least a positive solution in $S^1_0(\Omega)$  for any $0<\lambda<\lambda_1,$ where $\lambda_1$ is the first eigenvalue of $-\Delta_{\G}$ in $S^1_0(\Omega).$ It is also known that for $\lambda \geq \lambda_1,$ there is no positive solution of \eqref{BNLocal}. Before that, the problem \eqref{BNLocal} was studied by Citti \cite{Citti95} in the setting of the Heisenberg group using the explicit form of minimizers of subelliptic Sobolev inequality due to Jerison and Lee \cite{JL89}. However, in the context of Carnot groups, the explicit form of the Sobolev extremals is not known, posing a challenge for applying the Brezis-Nirenberg approach. To address this, Loiudice \cite{L07} utilised the exact behaviour of any extremal function of the Folland-Stein Sobolev inequality by means of the deep analysis performed by Bonfiglioli and Uguzzoni \cite{BU04}.  We also refer to several other interesting papers \cite{GV00, Lan03, BR17, GL92, BPV22, MBF16, MM07, BFP20a} and reference therein dealing with subelliptic nonlinear equations on stratified Lie groups. 

Regarding the multiplicity results of positive solutions to \eqref{BNLocal},  Loiudice \cite{L07} extended the seminal paper of  Cerami, Fortunato, and Struwe \cite{CFS84} in the setting of stratified Lie groups. She proved that problem \eqref{BNLocal} admit bifurcation from any eigenvalue $\lambda_j$ of $-\Delta_{\G}$ in the Folland-Stein Sobolev space $S^1_0(\Omega).$ More precisely, for $\lambda$ belonging to a suitable left neighborhood $(\lambda_j^*, \lambda_j)$ for each $\lambda_j,$ the number of nontrivial solutions is at least twice the multiplicity of $\lambda_j.$ Pucci \cite{P19} studied the existence and multiplicity results for quasilinear equations in the Heisenberg group. In \cite{KD20I}, Kassymov and Suragan established a multiplicity result for a subcritical subelliptic problem involving Hardy potential on stratified Lie groups.

 In \cite{GKR22} with Ruzhansky, we investigated an eigenvalue problem for $(-\Delta_{p,\mathbb{G}})^s$ on fractional Folland-Stein-Sobolev spaces $X_0^{s,p}(\Omega)$ using the variational method. In \cite{GKR22}, we also have established the compactness of fractional Folland-Stein-Sobolev embeddings on $\G$ which serve as a crucial tool to study nonlocal equations. For further exploration of the best constants of such fractional Folland-Stein-Sobolev inequalities, we recommend consulting \cite{GKR23, GGKR25}. In \cite{GKR24} we considered the Brezis-Nirenberg problem
\eqref{pro1intro} associated with a nonlocal subelliptic operator on $\G$ and prove the existence of at least one position solution in $X_0^{s,2}(\Omega)$ for $\lambda$ belongs to a suitable interval.  In this paper, we aim to continue this line of research by studying the bifurcation phenomena and the existence of multiple solutions for the nonlocal subelliptic Brezis-Nirenberg problem \eqref{pro1intro}. Moreover, we study the properties of the $(s,p)$-spectrum of $(-\Delta_{p,\mathbb{G}})^s$ characterizing the second eigenvalue $\lambda_2(\Omega)$. Our work extends the results from \cite{CFS84, FMS16} to the sub-Riemannian setting and from \cite{L07} to the nonlocal setting on stratified Lie groups. The main result of this paper is stated below.

\begin{theorem} \label{mainthm}
    Let $\G$ be a stratified Lie group of homogeneous dimension $Q$ and let  $ \Omega \subset \G$ be a bounded domain. Let $s \in (0, 1)$ with $Q>2s.$ Let $\lambda \in \mathbb{R}$ and $\lambda^*$ be the eigenvalue of problem 
    \begin{equation} \label{eignintro}
      \begin{cases} 
(-\Delta_{\mathbb{G}})^s u=  \lambda u \quad &\text{in}\quad \Omega, \\
u=0\quad & \text{in}\quad \mathbb{G}\backslash \Omega,
\end{cases}
  \end{equation} given by 
  \begin{equation} \label{eqq1.11}
      \lambda^*:= \min \{ \lambda_k: \lambda <\lambda_k \}
  \end{equation}
  with $m$ being its multiplicity. Assume that 
  \begin{equation} \label{eqq1.12}
      \lambda \in (\lambda^*- C_{2_s^*, Q} |\Omega|^{\frac{-2s}{Q}}, \lambda^*),
  \end{equation}
  where $C_{2_s^*, Q}$ is the best constant in the fractional critical Folland-Stein-Sobolev embedding.
    Then the problem \eqref{pro1intro} admits at least $m$ pair of nontrivial solutions $\{-u_{\lambda, i}, u_{\lambda, i}\}$ such that 
    $$\|u_{\lambda, i}\|_{X_0^{s,2}} \rightarrow 0\quad \text{as} \quad \lambda \rightarrow \lambda^*$$ for any $i=1,2,\ldots, m. $
\end{theorem}

The proof of Theorem \ref{mainthm} is based on an abstract result established by Bartolo, Benci, and Fortunato \cite{BBF83}. It relies on the linking argument, which is based on the eigenspace of $(-\Delta_{\G})^s,$ as well as a detailed analysis of the variational Dirichlet fractional subelliptic eigenvalue problem investigated in this paper and \cite{GKR22} along with fractional Sobolev embeddings on stratified Lie groups.

\begin{remark} It is desirable to obtain Theorem \ref{mainthm} for the more general (nonlinear) fractional $p$-sub-Laplacian, $(-\Delta_{\mathbb{G}, p})^s$ defined by 
\begin{equation} \label{fracps}
        (-\Delta_{\mathbb{G}, p})^su(x):= C(Q, s, p)  P.V. \int_{{\mathbb{G}} } \frac{|u(x)-u(y)|^{p-2}(u(x)-u(y))}{\left|y^{-1} x\right|^{Q+p s}} d y, \quad x \in {\mathbb{G}},
    \end{equation}
    where  $0<s<1<p<\infty$ with $Q>ps$. This will provide the complete resolution of bifurcation and multiplicity results for the subelliptic Brezis-Nirenberg problems studied in \cite{GKR22, GKR24}. Such an extension is highly nontrivial and necessitates a novel abstract critical point theorem that does not rely on linear spaces. Specifically, the proof of Theorem \ref{mainthm} depends on a linking argument constructed using the (linear) eigenspaces of $(-\Delta_\G)^s$. However, this approach is not applicable when $p \neq 2$, as the nonlinear operator $(-\Delta_{\mathbb{G}, p})^s$ lacks linear eigenspaces. In this context, it is worth noting the work of Perera, Squassina, and Yang \cite{PSY16}, which addresses bifurcation and multiplicity results for the critical fractional $p$-Laplacian in $\mathbb{R}^n$. Their approach involves a more general construction based on sublevel sets and a sequence defined using the $\mathbb{Z}_2$-cohomological index. This methodology is essential because the standard sequence of variational eigenvalues of $(-\Delta_{\mathbb{G}, p})^s$ does not provide sufficient information about the structure of the sublevel sets needed for the linking construction. Extending the techniques developed in \cite{PSY16} to subelliptic operators on stratified Lie groups would be a valuable direction for future research.
\end{remark}

Now, we turn our attention to the nonlocal $(s, p)$-eigenvalue problem. The study of nonlocal $(s, p)$-eigenvalue problems was initially explored by Lindgren and Lindqvist \cite{LL14}, Franzina and Palatucci \cite{FP14}, and Brasco and Parini \cite{BP16}. Since then, this topic has attracted significant attention. For instance, we refer to \cite{L05, L17, AA87, BP16, LL14, FP14}, among others, which have contributed to the advancement of the nonlinear eigenvalue problem. To the best of our knowledge, the investigation of eigenvalue problems in the subelliptic setting has been relatively limited. The earliest works on this subject date back to \cite{MS78, FP81}. Subsequent developments include studies on the $p$-sub-Laplacian on the Heisenberg group, as presented in \cite{HL08}. More recently, there has been increasing interest in eigenvalue problems involving subelliptic operators on stratified Lie groups. For further details, we refer to \cite{CC21, HL08, FL10, CC19, GU21a} and the references therein. The nonlinear nonlocal Dirichlet eigenvalue problem on stratified Lie groups was introduced in our previous paper \cite{GKR22} with Ruzhansky. Recall the eigenvalue problem for the fractional $p$-sub-Laplacian $(-\Delta_{p,{\mathbb{G}}})^s$ from \cite{GKR22}:
\begin{align} \label{sp eigen1}
	(-\Delta_{p,{\mathbb{G}}})^s u&=\lambda |u|^{p-2}u,~\text{in}~\Omega,\nonumber\\
	u&=0~\text{ in }~{\mathbb{G}}\setminus\Omega,
\end{align}
where $\Omega$ is bounded domain in ${\mathbb{G}}$ of homogeneous dimension $Q$ and $0<s<1<p<\infty$ with $Q>ps$. 
For recent developments on the fractional $p$-sub-Laplacian $(-\Delta_{p,{\mathbb{G}}})^s$, we refer to \cite{GKR22, GKR24, KD20, GKR23,FZ24, PP24, Pic22} and references therein. Notably, for $p = 2$, this operator reduces to the fractional sub-Laplacian $(-\Delta_\G)^s$.

In this paper, we extend the study of the $(s, p)$-spectrum of $(-\Delta_{\mathbb{G}, p})^s$, building on properties $(i)-(v)$ of Theorem \ref{ev-mainthmintro}  below established in \cite{GKR22}. We consolidate previous findings with new results to provide a comprehensive reference for the properties of $(s, p)$-eigenvalues and $(s, p)$-eigenfunctions. We believe these results will have applications beyond this work, guiding subelliptic variational problems on stratified Lie groups.
\begin{theorem}\label{ev-mainthmintro}
 	Let $0<s<1<p<\infty$ and let $\Omega$ be a bounded domain of a stratified Lie group $\mathbb{G}$ of homogeneous dimension $Q$. Then for $Q>ps$, the eigenvalues and eigenfunctions of problem \eqref{sp eigen1} associated  $(-\Delta_{p,{\mathbb{G}}})^s$ have the following properties:
	\begin{enumerate}[label=(\roman*)]
	\item 	The first eigenvalue $\lambda_1$ of \eqref{sp eigen1} is given by 
	\begin{equation}
		\lambda_1:=\lambda_1(\Omega)=\inf_{\overset{u \in X_0^{s,p}(\Omega)}{\|u\|_p=1} } \int_{\mathbb{G} \times \mathbb{G}}  \frac{|u(x)-u(y)|^p}{|y^{-1}x|^{Q+ps}}dxdy
	\end{equation} or equivalently, 
	\begin{equation} \label{eigva11}
		\lambda_1:=\inf_{u \in X_0^{s,p}(\Omega) \setminus \{0\}} \frac{ \int_{\mathbb{G} \times \mathbb{G}}  \frac{|u(x)-u(y)|^p}{|y^{-1}x|^{Q+ps}}dxdy}{\int_{\Omega} |u(x)|^p\,dx}
	\end{equation}

 \item  There exists a non-negative function $e_1 \in X_0^{s,p}(\Omega),$ the eigenfunction corresponding to $\lambda_1,$ attaining the minimum in \eqref{eigva11}, that is, $\|e_1\|_{L^p(\Omega)}=1$ and 
 $$\lambda_1=\int_{\mathbb{G} \times \mathbb{G}}  \frac{|e_1(x)-e_1(y)|^p}{|y^{-1}x|^{Q+ps}}dxdy.$$ In particular, the first eigenvalue $\lambda_1$ is principle.

		\item The first eigenvalue $\lambda_1$ of the problem \eqref{sp eigen1}  is simple and the corresponding eigenfunction $\phi_1$ is the only eigenfunction of constant sign, that is, if $u$ is an eigenfunction associated to an eigenvalue $\nu>\lambda_1(\Omega)$, then $u$ must be sign-changing.
  \item The first eigenvalue $\lambda_1$ of the problem \eqref{sp eigen1}  is isolated. 
  
  \item The eigenfunctions for positive eigenvalues to \eqref{sp eigen1} are bounded in $\mathbb{G}$.
  \item Assume that $(\Omega_j)$ is an increasing sequence of domains such that $\Omega=\cup_{j=1}^{\infty}\Omega_j$. Then $\lambda_1(\Omega_j)$ decreases to $\lambda_1(\Omega)$.
   \item The set of all $(s,p)$-eigenvalues, that is the spectrum $\sigma(s,p)$ to \eqref{sp eigen1} is closed.
   
		\item The set of eigenvalues of the problem \eqref{sp eigen1} consists of a sequence $(\lambda_n)$ with
  \begin{align}
      0<\lambda_1<\lambda_2 \leq \ldots \leq \lambda_n \leq \lambda_{n+1} \leq \ldots~\text{and }&
      \lambda_n \rightarrow \infty\,\,\,\text{as}\,\,n \rightarrow \infty.
  \end{align}
 \item Moreover, when $p=2$, for each $n \in \N$, the sequence of eigenvalues can be characterized as follows:
  \begin{equation}
      \lambda_{n+1}:= \min_{\overset{u \in \mathbb{P}_{n+1}}{\|u\|_{L^2(\Omega)}=1} } \int_{\mathbb{G} \times \mathbb{G}}  \frac{|u(x)-u(y)|^2}{|y^{-1}x|^{Q+2s}}dxdy,
  \end{equation}
or, equivalently
 \begin{equation}
	\lambda_{n+1}:= \min_{{u \in \mathbb{P}_{n+1}}\setminus\{0\}} \frac{ \int_{\mathbb{G} \times \mathbb{G}}  \frac{|u(x)-u(y)|^2}{|y^{-1}x|^{Q+2s}}dxdy}{\int_{\Omega} |u(x)|^2\,dx},
\end{equation}
  where 
  \begin{equation}
      \mathbb{P}_{n+1}:= \left\{ u \in X_0^{s,2}(\Omega): \int_{\mathbb{G} \times \mathbb{G}}  \frac{(u(x)-u(y))(e_j(x)-e_j(y))}{|y^{-1}x|^{Q+2s}}dxdy=0,\right\}
  \end{equation}
for all $j=1,2, \ldots, n$ with $0<\lambda_1<\lambda_2 \leq \ldots \leq \lambda_n \leq \lambda_{n+1} \leq \ldots\infty$. 
		\item For $p=2$ and for any $n \in \N$ there exists a function $e_{n+1} \in \mathbb{P}_{n+1}$,  the eigenfunction corresponding to $\lambda_{n+1},$ attaining the minimum in \eqref{eigva11}, that is, $\|e_{n+1}\|_{L^2(\Omega)}=1$ and 
 \begin{equation}
 	\lambda_{n+1}=\int_{\mathbb{G} \times \mathbb{G}}  \frac{|e_{n+1}(x)-e_{n+1}(y)|^2}{|y^{-1}x|^{Q+2s}}dxdy.
 \end{equation}
 \item For $p=2$, the sequence $(e_n)_{n \in \N}\subset X_0^{s,2}(\Omega)$ of eigenfunctions corresponding to $\lambda_n$ is an orthonormal basis of $L^2(\Omega)$ and an orthogonal basis of $X_0^{s,2}(\Omega).$
 \item Moreover, for $p=2$, each eigenvalue $\lambda_n$ has finite multiplicity. More precisely, if $\lambda_n$ satisfies 
 $$\lambda_{n-1} < \lambda_n = \cdot =\lambda_{n+m}<\lambda_{n+m+1}$$ for some $m \in \mathbb{N}\cup\{0\},$ then the set of all eigenfunctions corresponding to $\lambda_n$ agrees with 
 $$\text{span} \{e_n, \cdots, e_{n+m}\}.$$
	\end{enumerate}	
\end{theorem}

\begin{remark} The eigenvalues $\lambda_1(\Omega)$ and $\lambda_2(\Omega)$ can be estimated from below using the Faber-Krahn and Hong-Krahn-Szegö inequalities, respectively in the Euclidean case (see \cite{LL14, BP16}). However, in the context of stratified Lie groups, particularly the Heisenberg group $\mathbb{H}^n$, these inequalities are not directly applicable. They rely on symmetric decreasing rearrangements, a method central to the Euclidean setting that involves rearranging functions into radially symmetric, decreasing forms while preserving their $L^p$ norms and minimizing Dirichlet integrals. The non-commutative structure, anisotropy, and sub-Riemannian geometry of stratified Lie groups like $\mathbb{H}^n$ prevent the use of this technique, rendering classical approaches to these spectral inequalities ineffective (see \cite{FL}). Therefore, a completely new approach, independent of symmetric rearrangement arguments, is necessary and will be explored in future work.
     
\end{remark}

The structure of the paper is as follows. Section \ref{s2} introduces preliminary results, including fundamental notations and concepts related to stratified Lie groups, as well as the definition and properties of fractional Sobolev spaces within this framework. In Section \ref{s3}, we explore the subelliptic nonlocal $(s,p)$-eigenvalue problem and provide a detailed proof of Theorem \ref{ev-mainthmintro}. Finally, in Section \ref{s4}, we present the proof of Theorem \ref{mainthm}, focusing on bifurcation phenomena and the existence of multiple solutions for the nonlocal subelliptic Brezis-Nirenberg problem.

\section{Preliminaries: Stratified Lie groups and fractional Sobolev spaces}\label{s2}

In this section, we will review some essential tools related to stratified Lie groups and the fractional Sobolev spaces defined on them. There are multiple approaches to introducing the concept of stratified Lie groups, and interested readers can refer to various books and monographs such as \cite{FS82, BLU07, GL92, FR16, RS19,F75}.  

A Lie group $\mathbb{G}$ (on $\mathbb{R}^N$ with a Lie group law $\circ$) is said to be {\it homogeneous} if, for each $\lambda>0$, there exists an automorphism $D_{\lambda}:\mathbb{G} \rightarrow\mathbb{G}, $ called a  {\it dilation}, defined by $D_{\lambda}(x)=  (\lambda x^{(1)}, \lambda^2 x^{(2)},..., \lambda^{k}x^{(k)})$ for $x^{(i)} \in \mathbb{R}^{N_i},\,\forall\, i=1,2,..., k$ and $N_1 + N_2+ ... + N_k = N.$   We denote the Lie algebra associated with the Lie group $\G =(\mathbb{R}^N, \circ),$ that is, the Lie algebra of left-invariant (with respect to the group law $\circ$) vector fields on $\mathbb{R}^N$ by $\mathfrak{g}$. With $N_1$ the same as in the above decomposition of $\mathbb{R}^N$, let $X_1, ..., X_{N_1} \in \mathfrak{g} $  such that $X_i(0) = \frac{\partial}{\partial x_i}|_0$ for $i = 1, ..., N_1$. We make the following assumption: 

{\it The H\"ormander condition $rank(Lie\{X_1, ..., X_{N_1} \}) = N$
		holds for every $x \in \mathbb{R}^{N}$, that is,   the Lie algebra generated  by  $X_1, ..., X_{N_1}$ is whole $\mathfrak{g}$.}

With the above hypothesis, we call $\G =(\mathbb{R}^n, \circ, D_\lambda)$ a stratified Lie group (or a homogeneous Carnot group). Here $k$ is called the step of the stratified Lie group and $n:=N_1$ is the number of generators. The number $Q:=\sum_{i=1}^{i=k}iN_i$ is called the Homogeneous dimension of $\G.$ At times. we write $\lambda x$ to denote the dilation $D_{\lambda}x$.  The Haar measure on $\mathbb{G}$ is denoted by $d x$ and it is nothing but the usual Lebesgue measure on $\mathbb{R}^N.$   Let $\Omega$ be a Haar measurable subset of $\mathbb{G}$. Then $\mu(D_{\lambda}(\Omega))=\lambda^{Q}\mu(\Omega)$ where $\mu(\Omega)$ is the Haar measure of $\Omega$. 

 We would like to note here that in the literature, a {\it stratified} Lie group (or Carnot group) $\mathbb{G}$ is defined as a connected and simply connected Lie group whose Lie algebra $\mathfrak{g}$ is {\it stratifiable}. This means $\mathfrak{g}$ admits a vector space decomposition $\mathfrak{g} = \bigoplus_{i=1}^k \mathfrak{g}_i$ such that  $[\mathfrak{g}_1, \mathfrak{g}_i] = \mathfrak{g}_{i+1}$ for all $i=1.2,\ldots, k-1$ and  $[\mathfrak{g}_1, \mathfrak{g}_k] = \{0\}.$ 
 It is not difficult to see that any homogeneous Carnot group is a Carnot group according to the classical definition (see \cite[Theorem 2.2.17]{BLU07}). The opposite implication is also true; we refer to \cite[Theorem 2.2.18]{BLU07} for a detailed proof.  

The advantage of our operative definition of a stratified Lie group is that it is not only more convenient for analytic purposes but also enables us to gain a clearer understanding of the underlying group law.   Examples of stratified Lie groups include the Heisenberg group, more generally, $H$-type groups, the Engel group, and the Cartan group.  

For any $x,y\in\mathbb{G}$ the Carnot-Carath\'{e}odory distance is defined as
$$\rho_{cc}(x,y)=\inf\{l>0: ~\text{there exists an admissible}~ \gamma:[0,l]\rightarrow\mathbb{G} ~\text{with}~ \gamma(0)=x ~\text{\&}~ \gamma(l)=y \}.$$
We define $\rho_{cc}(x,y)=0$, if no such curve exists. We recall that an absolutely continuous curve $\gamma:[0,1]\rightarrow\mathbb{R}$ is said to be admissible, if there exist functions $c_i:[0,1]:\rightarrow\mathbb{R}$, for $i=1,2...,n$ such that
$${\dot{\gamma}(t)}=\sum_{i=1}^{n}c_i(t)X_i(\gamma(t))~\text{and}~ \sum_{i=1}^{i=n}c_i(t)^2\leq1.$$
Note that the functions $c_i$ may not be unique since the vector fields $X_i$ might not be linearly independent. Generally,  $\rho_{cc}$ is not a metric, but the H\"ormander condition for the vector fields $X_1,X_2,...X_{n}$ guarantees that $\rho_{cc}$ is indeed a metric. Consequently, the space $(\mathbb{G}, \rho_{cc})$ is classified as a Carnot-Carathéodory space.

	A continuous function $|\cdot|: \mathbb{G} \rightarrow \mathbb{R}^{+}$ is said to be a homogeneous quasi-norm on a homogeneous Lie group $\mathbb{G}$ if it is symmetric ($|x^{-1}| = |x|$ for all $x \in\mathbb{G}$), definite ($|x| = 0$ if and only if $x = 0$), and $1$-homogeneous ( $|D_\lambda x| = \lambda |x|$ for all $x \in\mathbb{G}$ and $\lambda>0$).

An example of a quasi-norm on \(\mathbb{G}\) is the norm defined as \(d(x):=\rho_{cc}(x, 0),\,\, x\in \mathbb{G}\), where \(\rho\) is a Carnot-Carathéodory distance associated with H\"ormander vector fields on \(\mathbb{G}\). It is known that all homogeneous quasi-norms on \(\mathbb{G}\) are equivalent. In this paper, we will use a left-invariant homogeneous distance \(d(x, y):=|y^{-1} \circ x|\) for all \(x, y \in \mathbb{G}\), which is derived from the homogeneous quasi-norm on \(\mathbb{G}\). The quasi-ball of radius $r$ centered at $x\in\mathbb{G}$ with respect to the quasi-norm $|\cdot|$ is defined as
\begin{equation}\label{d-ball}
	B(x, r)=\left\{y \in \mathbb{G}: \left|y^{-1} \circ x\right|<r\right\}.
\end{equation}

As a consequence of H\"ormander hypoelliptic condition,
the sub-Laplacian (or Horizontal Laplacian)  on $\mathbb{G}$ is defined as
\begin{equation}\label{d-sub-lap}
	\Delta_{\mathbb{G}}:=X_{1}^{2}+\cdots+X_{n}^{2}
\end{equation}
is a hypoelliptic operator. 
The horizontal gradient on $\mathbb{G}$ is defined as
\begin{equation}\label{d-h-grad}
	\nabla_{\mathbb{G}}:=\left(X_{1}, X_{2}, \cdots, X_{n}\right).
\end{equation}

For $s \in (0, 1)$ and $p\in(1,\infty)$,  we define the fractional $p$-sub-Laplacian, $(-\Delta_{\mathbb{G}})^s$ as 
\begin{equation}
        (-\Delta_{p,\mathbb{G}})^s(u)(x):= C(Q,p, s)  P.V. \int_{{\mathbb{G}} } \frac{|u(x)-u(y)|^{p-2}(u(x)-u(y))}{\left|y^{-1} x\right|^{Q+p s}} d y, \quad x \in {\mathbb{G}},
    \end{equation}
     where $|\cdot|$ is a homogeneous norm on the stratified Lie group $\G,$ $C(Q,s)$ is the positive normalization constant depending only on the homogeneous dimension $Q$ and $s.$ For $p=2$, we define the fractional sub-Laplacian as
     \begin{equation}
        (-\Delta_{\mathbb{G}})^s(u)(x):= C(Q, s)  P.V. \int_{{\mathbb{G}} } \frac{(u(x)-u(y))}{\left|y^{-1} x\right|^{Q+2 s}} d y, \quad x \in {\mathbb{G}},
    \end{equation}

The simplest example of a stratified Lie group is the Heisenberg group $\mathbb{H}^N$ with the underlying manifold $\mathbb{R}^{2N+1}:=\mathbb{R}^N\times\mathbb{R}^N\times\mathbb{R}$ for $N\in \mathbb{N}$. For $(x, y, t),(x', y', t')\in \mathbb{H}^N$ the multiplication in $\mathbb{H}^N$ is given by

\begin{equation*}
	(x, y, t) \circ(x', y', t')=(x+x', y+y', t+t'+2 (\langle x',y\rangle)-\langle x,y'\rangle),
\end{equation*}
where $(x, y, t), (x', y', t') \in \mathbb{R}^{N}\times \mathbb{R}^{N}\times \mathbb{R}$ and $\langle\cdot,\cdot\rangle$ represents the inner product on $\mathbb{R}^N$. The homogeneous structure of  the Heisenberg group $\mathbb{H}^N$ is provided by the following dilation, for $\lambda>0,$
\begin{equation*}
	D_{\lambda}(x,y,t)=(\lambda x, \lambda y, \lambda^2 t).
\end{equation*}
the  homogeneous dimension $Q$ of $\mathbb{H}^N$ is given by $2N+2:= N+N+2$ while the topological dimension of $\mathbb{H}^N$ is $2N+1.$
The left-invariant vector fields $\{X_i,Y_i\}_{i=1}^N$ defined below form a basis for the Lie algebra corresponding to the Heisenberg group $\mathbb{H}^N$:
\begin{align}
	X_i&=\frac{\partial}{\partial x_i}+2y_i\frac{\partial}{\partial t};
	Y_i=\frac{\partial}{\partial y_i}-2x_i\frac{\partial}{\partial t}~\text{and}~
	T=\frac{\partial}{\partial t}, ~\text{for}~i=1, 2,..., N.
\end{align}
It is easy to see that $[X_i,Y_i]=-4T$ for $i=1,2,...,N$ and $$[X_i,X_j]=[Y_i,Y_j]=[X_i,Y_j]=[X_i,T]=[Y_j,T]=0$$ for all $i\neq j$ and these vector fields satisfy the H\"ormander rank condition. Consequently, the sub-Laplacian on $\mathbb{H}^N$ is given by 
$$\Delta_{\mathbb{H}^N}:=\sum_{i=1}^N (X_i^2+Y_i^2).$$

We now define the notion of fractional Sobolev-Folland-Stein spaces, which will be useful for our study. For complete details and comparisons with other definitions of fractional Sobolev spaces on a stratified Lie group $\G$, we refer to \cite{GKR22}. For an open subset $\Omega \subset \G,$ we define the Gagliardo semi-norm, denoted by $[u]_{s, p, \G},$  as
\begin{equation}
	[u]_{s, p,\G}:=\left(\int_{\Omega} \int_{\Omega} \frac{|u(x)-u(y)|^{p}}{\left|y^{-1} x\right|^{Q+ps}} d xd y\right)^{\frac{1}{p}}<\infty,
\end{equation}
  for $0<s<1<p<\infty.$ The fractional Sobolev space $W^{s,p}(\G)$ on stratified groups is defined as 
\begin{equation}
	W^{s,p}(\G):=\{u\in L^{p}(\G): [u]_{s, 2,\G}<\infty\},
\end{equation}
endowed with the norm 
\begin{equation}
	\|u\|_{W^{s,p}(\G)}=\|u\|_{L^p(\G)}+[u]_{s,p,\G}.
\end{equation}
In a similar way, 
given an open subset  $\Omega \subset {\mathbb{G}},$ we  define the fractional Sobolev-Folland-Stein type $X_0^{s,p}(\Omega)$  as the closure of $C_c^{\infty}(\Omega)$ with respect to the  norm $$\|u\|_{X_0^{s,p}(\Omega)}=\|u\|_{L^p(\Omega)}+[u]_{s, p,\mathbb{G}}.$$
 It is also known that the space $X_0^{s,p}(\Omega)$ can also be defined as a closure of $C_c^{\infty}(\Omega)$ with respect to the homogeneous norm $[u]_{s, p, \mathbb{G}}.$ Moreover,  $X_0^{s,p}(\Omega)$ can be represented as 
$$X_0^{s,p}(\Omega):= \{u \in W^{s,p}(\G): u =0 \,\,\text{on}\,\,\, \G \backslash \Omega  \}.$$ For any $\varphi\in X_0^{s,p}(\Omega)$, we have
\begin{equation}
	\langle\left(-\Delta_{{p, \mathbb{G}}}\right)^s u,\varphi\rangle= \int_{\mathbb{G} \times \mathbb{G}} \frac{|u(x)-u(y)|^{p-2}(u(x)-u(y))(\varphi(x)-\varphi(y))}{\left|y^{-1} x\right|^{Q+p s}} d xd y.
\end{equation}
For $p=2$, the space $X_0^{s,2}(\Omega)$ is a Hilbert space with the inner product given by
\begin{equation}\label{eq3.6}
	\langle u,\varphi\rangle_{X_0^{s,2}(\Omega)}= \int_{\mathbb{G} \times \mathbb{G}} \frac{(u(x)-u(y))(\varphi(x)-\varphi(y))}{\left|y^{-1} x\right|^{Q+2s}} dxd y.
\end{equation}

Now, we recall a result about the inclusion of fractional Sobolev spaces  $X_0^{s,p}(\Omega)$ into the Lebesgue spaces proved in [Theorem 1.1, \cite{GKR22}].  For a study of the best constants of such inequalities and the Fractional Morrey-Sobolev type inequalities on stratified Lie groups, the readers can consult \cite{GKR23, GGKR25} and references therein.

\begin{theorem} \label{l-3} Let $\mathbb{G}$ be a stratified Lie group of homogeneous dimension $Q$, and let $\Omega\subset\mathbb{G}$ be an open set.
Let $0<s<1\leq p<\infty$ with $\frac{Q}{s}>p.$  Then the fractional Sobolev space $X_0^{s,p}(\Omega)$ is continuously embedded in $L^r(\Omega)$ for $p\leq r\leq p_s^*:=\frac{pQ}{Q-ps}$, that is, there exists a positive constant $C=C(Q,p,s, \Omega)$ such that for all $u\in X_0^{s,p}(\Omega)$, we have
\begin{equation} \label{emG}
   \|u\|_{L^r(\Omega)}\leq C \|u\|_{X_0^{s,p}(\Omega)}.
\end{equation}

Moreover, if $\Omega$ is bounded, then the following embedding
\begin{align} \label{compactemG}
    X_0^{s,p}(\Omega) \hookrightarrow L^r(\Omega)
\end{align}
is continuous for all $r\in[1,p_s^*]$ and is compact for all $r\in[1,p_s^*)$.
\end{theorem}

\section{Eigenvalues of the fractional $p$-sub-Laplacian on stratified Lie groups} \label{s3}
In this section, we discuss the fundamental spectral properties for the nonlocal Dirichlet eigenvalue problem associated with the fractional $p$-sub-Laplacian (see \cite{GKR22}).
\begin{align} \label{sp eigen}
	(-\Delta_{p,{\mathbb{G}}})^s u&=\lambda |u|^{p-2}u,~\text{in}~\Omega,\nonumber\\
	u&=0~\text{ in }~{\mathbb{G}}\setminus\Omega.
\end{align}
and the fractional sub-Laplacian
\begin{align} \label{pro1.3}
		(-\Delta_{{\mathbb{G}}})^s u&=\lambda u,~\text{in}~\Omega,\nonumber\\
		u&=0~\text{ in }~{\mathbb{G}}\setminus\Omega,
	\end{align} 
    
where $\Omega$ is a bounded domain in ${\mathbb{G}}$ of homogeneous dimension $Q$ and $0<s<1<p<\infty$ with $Q>ps$. Note that for $p=2$, the problem \eqref{sp eigen} reduces to the problem \eqref{pro1.3}. We say that $u\in X_0^{s,p}(\Omega)$ is a non-trivial weak solution (eigenfunction) to \eqref{sp eigen} if, 
\begin{equation}\label{ev-soln}
	\langle \mathcal{A}(u),v\rangle=\lambda\int_{\Omega} |u|^{p-2}uv dx~\forall\,v\in X_0^{s,p}(\Omega),
\end{equation}
where the map $\mathcal{A}:X_0^{s,p}(\Omega)\rightarrow (X_0^{s,p}(\Omega))^*$ is given by
\begin{equation}\label{map A}
	\langle \mathcal{A}(u),v\rangle=\int_{\mathbb{G} \times \mathbb{G}}  \frac{|u(x)-u(y)|^{p-2}(u(x)-u(y))(v(x)-v(y))}{|y^{-1}x|^{Q+ps}}dxdy,~\forall\,u,v \in X_0^{s,p}(\Omega).
\end{equation}
Define the space $\mathcal{M}$ and the $C^1$ energy functional $I$ as
\begin{equation}\label{M}
	\mathcal{M}=\{u\in X_0^{s,p}(\Omega): p\|u\|_p=1\} \text{ and }I(u)=\frac{1}{p}\int_{\mathbb{G}\times \mathbb{G}}\frac{|u(x)-u(y)|^p}{|y^{-1}x|^{Q+ps}}dxdy.
\end{equation}
Then the eigenfunctions of \eqref{sp eigen} coincides with the critical points of $I$ on the space $\mathcal{M}$. We call $(\lambda, u)$, an eigenpair, where $\lambda$ is being referred to as eigenvalue. 
 
With these basic notations established, we are now prepared to prove Theorem \ref{ev-mainthmintro} for the eigenvalue problems \eqref{pro1.3} and \eqref{sp eigen}. For the sake of convenience, we restate the theorem here. 
 \begin{theorem}\label{ev-mainthm} Let $0<s<1<p<\infty$ and let $\Omega$ be a bounded domain of a stratified Lie group $\mathbb{G}$ of homogeneous dimension $Q$. Then for $Q>ps$, the eigenvalues and eigenfunctions of problem \eqref{sp eigen} associated  $(-\Delta_{p,{\mathbb{G}}})^s$ have the following properties:
	\begin{enumerate}[label=(\roman*)]
	\item 	The first eigenvalue $\lambda_1$ of \eqref{sp eigen} is given by 
	\begin{equation}\label{eigvaeq}
		\lambda_1:=\lambda_1(\Omega)=\inf_{\overset{u \in X_0^{s,p}(\Omega)}{\|u\|_p=1} } \int_{\mathbb{G} \times \mathbb{G}}  \frac{|u(x)-u(y)|^p}{|y^{-1}x|^{Q+ps}}dxdy
	\end{equation} or equivalently, 
	\begin{equation} \label{eigva}
		\lambda_1:=\inf_{u \in X_0^{s,p}(\Omega) \setminus \{0\}} \frac{ \int_{\mathbb{G} \times \mathbb{G}}  \frac{|u(x)-u(y)|^p}{|y^{-1}x|^{Q+ps}}dxdy}{\int_{\Omega} |u(x)|^p\,dx}
	\end{equation}
		
\item  There exists a non-negative function $e_1 \in X_0^{s,p}(\Omega),$ the eigenfunction corresponding to $\lambda_1,$ attaining the minimum in \eqref{eigva}, that is, $\|e_1\|_{L^p(\Omega)}=1$ and 
 $$\lambda_1=\int_{\mathbb{G} \times \mathbb{G}}  \frac{|e_1(x)-e_1(y)|^p}{|y^{-1}x|^{Q+ps}}dxdy.$$ In particular, the first eigenvalue $\lambda_1$ is principle and there exists a constant $C=C(Q,p,s)>0$ such that $\lambda_1(\Omega)\geq C |\Omega|^{-\frac{ps}{Q}}$.

		\item The first eigenvalue $\lambda_1$ of the problem \eqref{pro1.3} is simple and the corresponding eigenfunction $\phi_1$ is the only eigenfunction of constant sign, that is, if $u$ is an eigenfunction associated to an eigenvalue $\nu>\lambda_1(\Omega)$, then $u$ must be sign-changing.
  \item The first eigenvalue $\lambda_1$ of the problem \eqref{pro1.3} is isolated.
  \item The eigenfunctions for positive eigenvalues to \eqref{sp eigen} are bounded in $\mathbb{G}$.
  \item Assume that $(\Omega_j)$ is an increasing sequence of domains such that $\Omega=\cup_{j=1}^{\infty}\Omega_j$. Then $\lambda_1(\Omega_j)$ decreases to $\lambda_1(\Omega)$.
   \item The set of all eigenvalues, that is the spectrum $\sigma(s,p)$ to \eqref{sp eigen} is closed.
   
		\item The set of eigenvalues of the problem \eqref{sp eigen} consists of a sequence $(\lambda_n)$ with \begin{align}
     0<\lambda_1<\lambda_2 \leq \ldots \leq \lambda_n \leq \lambda_{n+1} \leq \ldots~\text{and }&
      \lambda_n \rightarrow \infty\,\,\,\text{as}\,\,n \rightarrow \infty.
  \end{align}
 \item Moreover, when $p=2$, for each $n \in \N$, the sequence of eigenvalues can be characterized as follows:
  \begin{equation}\label{eq3.10}
      \lambda_{n+1}:= \min_{\overset{u \in \mathbb{P}_{n+1}}{\|u\|_{L^2(\Omega)}=1} } \int_{\mathbb{G} \times \mathbb{G}}  \frac{|u(x)-u(y)|^2}{|y^{-1}x|^{Q+2s}}dxdy,
  \end{equation}
or, equivalently
 \begin{equation}\label{eq3.10-a}
	\lambda_{n+1}:= \min_{{u \in \mathbb{P}_{n+1}}\setminus\{0\}} \frac{ \int_{\mathbb{G} \times \mathbb{G}}  \frac{|u(x)-u(y)|^2}{|y^{-1}x|^{Q+2s}}dxdy}{\int_{\Omega} |u(x)|^2\,dx},
\end{equation}
  where 
  \begin{equation} \label{eq3.11}
      \mathbb{P}_{n+1}:= \left\{ u \in X_0^{s,2}(\Omega): \int_{\mathbb{G} \times \mathbb{G}}  \frac{(u(x)-u(y))(e_j(x)-e_j(y))}{|y^{-1}x|^{Q+2s}}dxdy=0,\right\} \end{equation}
 for all $j=1,2, \ldots, n$ with $0<\lambda_1<\lambda_2 \leq \ldots \leq \lambda_n \leq \lambda_{n+1} \leq \ldots\infty$. 
		\item For $p=2$ and for any $n \in \N$ there exists a function $e_{n+1} \in \mathbb{P}_{n+1}$,  the eigenfunction corresponding to $\lambda_{n+1},$ attaining the minimum in \eqref{eigva}, that is, $\|e_{n+1}\|_{L^2(\Omega)}=1$ and 
 \begin{equation}\label{eq x}
 	\lambda_{n+1}=\int_{\mathbb{G} \times \mathbb{G}}  \frac{|e_{n+1}(x)-e_{n+1}(y)|^2}{|y^{-1}x|^{Q+2s}}dxdy.
 \end{equation}
 \item For $p=2$, the sequence $(e_n)_{n \in \N}\subset X_0^{s,2}(\Omega)$ of eigenfunctions corresponding to $\lambda_n$ is an orthonormal basis of $L^2(\Omega)$ and an orthogonal basis of $X_0^{s,2}(\Omega).$
 \item Moreover, for $p=2$, each eigenvalue $\lambda_n$ has finite multiplicity. More precisely, if $\lambda_n$ satisfies 
 $$\lambda_{n-1} < \lambda_n = \cdot =\lambda_{n+m}<\lambda_{n+m+1}$$ for some $m \in \mathbb{N}\cup\{0\},$ then the set of all eigenfunctions corresponding to $\lambda_n$ agrees with 
 $$\text{span} \{e_n, \cdots, e_{n+m}\}.$$
	\end{enumerate}	
\end{theorem}
\begin{proof}
	The proof of $(i)-(v)$ are straight forward from \cite{GKR22}. We will prove the rest of the claims.\\
	\underline{\textbf{{Proof of (vi):}}} Clearly, $\lambda_1(\Omega_j)$ is decreasing as $j$ increases. Thus the $\lim_{j \rightarrow+\infty}\lambda_1(\Omega_j)$ exists. Since $\lambda_1(\Omega_j)$ coincides with the Raleigh quotient $$\mathcal{R}(u)=\inf\limits_{u\in C_c^{\infty}(\Omega)}\cfrac{\int_{\mathbb{G}\times\mathbb{G}}\frac{|u(x)-u(y)|^p}{|y^{-1}x|^{Q+ps}}dxdy}{\int_{\Omega}|u(x)|^p dx},$$
	then for each $\epsilon>0$ there exists $\phi\in C_c^{\infty}(\Omega)$ such that $$\cfrac{\int_{\mathbb{G}\times\mathbb{G}}\frac{|u(x)-u(y)|^p}{|y^{-1}x|^{Q+ps}}dxdy}{\int_{\Omega}|u(x)|^p dx}<\lambda_1(\Omega)+\epsilon.$$
	Choose, $j_0$ large enough so that $supp(\phi)\subset\Omega_{j_0}$ and hence $\phi$ serves as a test function in Raleigh quotient. Thus, for all $j\geq j_0$, we have
	$$\lambda_1(\Omega_j)<\lambda_1(\Omega)+\epsilon.$$ This completes the proof.
	
	\underline{\textbf{{Proof of (vii):}}}
	    Suppose $\lambda\in\bar{\sigma(s,p)}$. Then, there exists a sequence of eigenvalues $(\lambda_n)$ of the problem $\eqref{sp eigen}$ such that $\lambda_n \rightarrow \lambda$. Denote $u_n$ as an eigenfunction corresponding to the eigenvalue $\lambda_n$ for all $n \in \mathbb{N}$. Thus, using the boundedness of $(\lambda_n)$ and the fact that $\int_\Omega |u_n|^p dx=1$, we put $v=u_n$ as the test function in \eqref{ev-soln} to obtain
	    	\begin{equation*}
	    	\lambda_n=\|u_n\|_{X_0^{s,p}(\Omega)}^p.
	    \end{equation*}
    Therefore, we get $(u_n)$ is bounded in the reflexive space $X_0^{s,p}(\Omega)$ asserting a weakly convergent subsequence, still denoted by $(u_n)$ such that $u_n\rightharpoonup u\in  X_0^{s,p}(\Omega)$. By Theorem \ref{l-3}, we conclude $(u_n)$ converges strongly to $u\in L^q(\Omega)$ for all $q\in[1,p_s^*)$. Therefore, using the Ho\"lder inequality, we get
    \begin{align}\label{eq 3.12}
    	|\langle \mathcal{A}(u_n), u_n-u \rangle|&= |\lambda_n \int_\Omega |u_n|^{p-2}u_n(u_n-u) dx|\nonumber\\
    	&\leq\lambda_n\|u_n-u\|_p\|u_n\|_p^{p-1}.
    \end{align}
    Thus, we have
    \begin{align}\label{eq 3.13}
    	\lim_{n\rightarrow\infty}\langle \mathcal{A}(u_n), u_n-u \rangle&=0.
    \end{align}
   Again, using $u_n \rightharpoonup u$ in $X_0^{s,p}(\Omega)$, we get
	    \begin{equation}\label{eq 3.14}
	    	\lim_{n\rightarrow\infty}\langle \mathcal{A}(u), u_n-u \rangle=0.
	    \end{equation}
    Thus, from \eqref{eq 3.13} and \eqref{eq 3.14}, we obtain
    \begin{equation}\label{eq 3.15}
    	\lim_{n\rightarrow\infty}\langle \mathcal{A}(u_n)-\mathcal{A}(u), u_n-u \rangle=0.
    \end{equation}
    Recall the following scalar inequality from \cite[Lemma B.4]{BP16}.
\begin{equation}\label{ineq 3.16}
	(|a|^{p-2}a-|b|^{p-2}b)(a-b) \geq \begin{cases}
		2^{2-p} |a-b|^p, &\text{ if } p \geq2,\\
		(p-1)|b-a|^2(|b|^2+|a|^2)^{\frac{p-2}{2}}, &\text{ if } 1<p<2.
	\end{cases}
\end{equation}
Therefore, for $p\geq 2$, using \eqref{ineq 3.16}, we obtain
\begin{equation}\label{eq 3.17}
	\|u_n-u\|^p\leq c(p)\langle \mathcal{A}(u_n)-\mathcal{A}(u), u_n-u \rangle
\end{equation}	    
On the other hand, for $1<p<2$, we have
\begin{align}\label{eq 3.18}
	\|u_n-u\|^p&\leq c(p)^{\frac{p}{2}}\langle \mathcal{A}(u_n)-\mathcal{A}(u), u_n-u \rangle^{\frac{p}{2}}(\|u_n\|^p+\|u\|^p)^{\frac{2-p}{2}}\nonumber\\
	&\leq c(p)^{\frac{p}{2}}\langle \mathcal{A}(u_n)-\mathcal{A}(u), u_n-u \rangle^{\frac{p}{2}}(\|u_n\|^{\frac{p(2-p)}{2}}+\|u\|^p)^{\frac{p(2-p)}{2}}\nonumber\\
			&\leq C^{\frac{p}{2}}\langle \mathcal{A}(u_n)-\mathcal{A}(u), u_n-u \rangle^{\frac{p}{2}}
\end{align}	     
Therefore, from \eqref{eq 3.17} and \eqref{eq 3.18}, we deduce that	    
\begin{equation}\label{eq 3.19}
	 \lim_{n\rightarrow\infty}\|u_n-u\|^p=0,
\end{equation}   
	 proving $u_n\rightarrow u$ strongly in $\X$. Since $(\lambda_n, u_n)$ is eigenpair, we have
	 \begin{equation}\label{eq 3.20}
	 		 	 \langle \mathcal{A}(u_n),v\rangle=\lambda_n \int_\Omega |u_n|^{p-2}u_n v dx, \,\, \text{ for all }   v\in X_0^{s,p}(\Omega).
	 \end{equation}
 Since $\Omega$ is bounded, using $u_n\rightharpoonup u\in  X_0^{s,p}(\Omega)$, $u_n\rightarrow u\in L^q(\Omega)$ for all $q\in[1,p_s^*)$ and the H\"older inequality, we get
 \begin{align}
 	\int_\Omega \left( |u_n|^{p-2}u_n-|u|^{p-2}u \right)v dx \rightarrow 0~\text{ and } \langle \mathcal{A}(u_n),v\rangle\rightarrow \langle \mathcal{A}(u),v\rangle \,\, \text{ as } n\rightarrow \infty.
 \end{align}
	 Hence, passing to the limit as $n \rightarrow \infty$ in \eqref{eq 3.20}, we obtain
	 \begin{align}\label{eq 3.21}
	 	\langle \mathcal{A}(u),v\rangle&=\lim_{n\rightarrow\infty}\langle \mathcal{A}(u_n),v\rangle=\lambda_n \int_\Omega |u_n|^{p-2}u_n v dx\nonumber\\
	 	 &=\lambda \int_\Omega |u|^{p-2}u v dx \,\, \text{ for all }   v\in X_0^{s,p}(\Omega). 
	 \end{align}
	 Also, $\|u\|_p=\lim_{n\rightarrow\infty}\|u_n\|_p=1$. Hence, $(\lambda, u)$ is an eigenpair to \eqref{sp eigen}.\\
	 \underline{\textbf{{Proof of (viii):}}} We employ the \textit{Ljusternik-Schnirelmann principle} \cite{S88} for the existence of a sequence of eigenvalues. We first recall some fundamental properties of the Krasnoselskii genus \cite{S88}, Proposition 2.3]. Let $\Sigma=\{A \subset X_0^{s,p}(\Omega) \setminus \{0\} \mid A \text{ is compact and }A=-A \}$. For every $A\in\Sigma$, the genus of $A$ is defined as
	    \begin{equation}
	    	\gamma(A)=\min\{k\in\mathbb{N}:\,\exists\,\phi\in C(A,\mathbb{R}^k\setminus\{0\}) \text{ such that } \phi(x)=\phi(-x)\}
	    \end{equation}
	We denote $\gamma(A)=\infty$ if the minimum does not exist. We will use the property: {\it Let $Z$ be a subspace of $\X$ with codimension $k$ and $\gamma(A)>k$, then $A \cap Z \neq \emptyset$}. To apply the \textit{Ljusternik-Schnirelmann principle} on $I$, we note that $I(0)=0$, $I$ is even and of class $C^1$. Moreover, $I$ is bounded from below on $X_0^{s,p}(\Omega)$.\\
    \textbf{Claim:} $I$ satisfies the Palais-Smale condition on $\mathcal{M}=\{u\in X_0^{s,p}(\Omega): p\|u\|_p=1\}$.\\
    \textbf{Proof of the Claim:} Let $(u_n)\subset\mathcal{M}$ be a sequence such that 
    $I(u_n)\leq C$ for all $n\in\mathbb{N}$ and $\lim_{n\rightarrow\infty}\langle I'(u_n)-\tau_nJ'(u_n),\varphi\rangle=0~\forall\,\varphi\in\X$, where $\tau_n=\frac{\langle I'(u_n), u_n\rangle}{\langle J'(u_n), u_n\rangle}$ and $J(u)=\frac{1}{p}\|u\|_{L^p(\Omega)}^p$. We need to show $(u_n)$ possesses a strongly convergent subsequence in $\mathcal{M}$. Indeed, we have $$0\leq \|u_n\|^p_{\X}\leq pI(u_n)\leq Cp,~\forall\, n\in\mathbb{N},$$
    implying the boundedness of $(u_n)$ in the reflexive space $\X$. Therefore, there exists $u\in \X$ such that $u_n\rightharpoonup u$ weakly in $\X$. From Theorem \ref{l-3}, we conclude $u_n\rightarrow u$ strongly in $L^q(\Omega)$ for all $q\in[1,p_s^*)$ and $u_n\rightarrow u$ a.e. in $\Omega$. Now, we proceed with similar arguments as in the \textit{proof of (vii)} to conclude that $u\neq0$ and $\|u_n-u\|_{\X}\rightarrow 0$ as $n\rightarrow\infty$. Thus, the claim is proved. 
    
    Therefore, from \textit{Ljusternik-Schnirelmann theory} \cite{S88} we conclude that there exists a sequence of positive critical values
	    \begin{equation}\label{eq 3.25}
	    	0<\lambda_1 \leq \lambda_2 \leq \cdots \leq \lambda_n \leq\cdots,
	    \end{equation}
	    where
	    $$ \lambda_n= \inf_{A \in \Sigma_n} \sup_{u \in A} pI(u) \text{ and } \Sigma_n= \{ A \in \Sigma : \gamma(A) \geq n \}.$$
	   We now prove $\lambda_n \rightarrow \infty$ as $n \rightarrow \infty$. Indeed, using the reflexivity and separability of $\X$, there exist sequences $(\nu_j)$ in $\X$ and $(\nu_j^*)$ in $(\X)^*$ such that 	   $X_0^{s,p}(\Omega)=\overline{\text{span}}\{\nu_j:j\in\mathbb{N}\}$ and $(X_0^{s,p}(\Omega))^*=\overline{\text{span}}\{\nu_j^*:j\in\mathbb{N}\}$. Moreover, $\langle \nu_i,\nu_j^* \rangle= \delta_{ij}$, for all $i,j=1,2,...,$. Define,
	    \begin{equation*}
	    	X_n=\overline{\mbox{span}}\{\nu_{1}, \nu_{2}, \cdots, \nu_n\} \text{ and } Z_n=\overline{\mbox{span}}\{\nu_{n}, \nu_{n+1}, \cdots\}.
	    \end{equation*}
    Choose, $A\in \Sigma_n$. Thus we have $\gamma(A)\geq n$ and hence using the properties of genus we derive $A\cap Z_n\neq\emptyset$. Therefore, we claim that $\lim_{n\rightarrow\infty}\alpha_n=\infty$, where
    $$\alpha_n=\inf_{A \in \Sigma_n} \sup_{u \in A\cap Z_n} pI(u).$$
    Suppose, the claim fails. Then there exists a sequence $(u_n)$ in $ A\cap Z_n$ such that $\int_{\Omega}|u_n|^p dx=1~\text{ and }~0\leq\alpha_n\leq pI(u_n) <\infty,\ n \in \mathbb{N}$ which implies that $(u_n)$ is bounded. Therefore, there exists a subsequence, still denoted by $(u_n)$ such that $u_n\rightharpoonup u \in\X$ weakly and $\int_{\Omega}|u|^p dx=1$. Therefore, we get $\langle \nu_j^*,u\rangle=\lim_{n\rightarrow\infty}\langle \nu_j^*,u_n\rangle=0$. This creates a contradiction due to the fact
    $$0=\int_{\Omega}|u|^p dx=\int_{\Omega}|u_n|^p dx=1.$$
    Thus, $\alpha_n\rightarrow\infty$ as $n\rightarrow\infty$. Note that $\alpha_n\leq\lambda_n$ and hence we get $\lambda_n\rightarrow\infty$ as $n\rightarrow\infty$.\\

 	It remains to prove that $\lambda_2>\lambda_1$. We follow Brasco and Parini \cite{BP16}. Using \eqref{M}, we define $$\Gamma_1(\Omega)=\left\{\phi:\mathbb{S}^1\to \mathcal{M}\, :\, \phi \mbox{ is odd and continuous}\right\}$$
 	and
 	\begin{equation}\label{eq3.27}
 		\lambda_2(\Omega)=\inf_{\phi\in \Gamma_1(\Omega)} \max_{u\in\mathrm{Im}(\phi)} \|u\|^p_{\X}.
 	\end{equation}
 	We first show that $\lambda_2$ given by \eqref{eq3.27} is an eigenvalue of \eqref{sp eigen}.\\
 	{\textbf{Claim:}} $\lambda_2(\Omega)$ is an eigenvalue of \eqref{sp eigen}.\\
 We show that $\lambda_2$ is a mountain pass type critical value for $I$ over $\mathcal{M}$. We first verify that $I$ satisfies the Palais-Smale condition on $\mathcal{M}$. Let $(u_n)$ be a PS sequence in $\mathcal{M}$, that is,
 		\begin{equation}\label{ps I}
 			I(u_n)\leq c~~ \text{ and }~~ \lim_{n\to\infty}\left\|DI(u_n)_{\restriction T_{u_n}\mathcal{M}}\right\|_*=0,
 		\end{equation}
 		where $DI(u_n)$ and $T_{u_n} \mathcal{M}$ denote the differential of $I$ at $u_n$ and the tangent space to $\mathcal{M}$ at $u_n$, which is defined as
 		\begin{equation}
 			T_{u_n} \mathcal{M}=\left\{\varphi\in \X\, :\, \int_\Omega |u_n|^{p-2}\, u_n\, \varphi\, dx=0\right\}.
 		\end{equation}
 	From, \eqref{ps I}, we have, $u_n$ is bounded. Therefore, using Theorem \ref{l-3} we get $u\in \X$ such that $u_n\rightharpoonup u$ weakly in $\X$ and $u_n\rightarrow u$ strongly in $L^p(\Omega)$. On using $u_n\in\mathcal{M}$, we have $u\in\mathcal{M}$. Moreover, the sequence $(\tau_n)$ converges to $1$, where $\tau_n$ is defined by
 	\begin{equation}\label{tau}
 		\tau_n := \int_\Omega |u_n|^{p-2}\, u_n\, u\, dx.
 	\end{equation}
 	Again, from \eqref{ps I}, we derive that there exists a sequence of positive numbers $(\delta_n)$ such that
 		\begin{equation}\label{ps2}
 		\delta_n\rightarrow0 \text{ and }	\left|\langle DI(u_n), \varphi\rangle\right|\leq \delta_n\, \|\varphi\|_{\X},~\forall\,\varphi\in T_{u_n} \mathcal{M}.
 		\end{equation}
 		We define the sequence $(v_n)$ as
 		\begin{equation*}
 			v_n=\tau_n\, u_n-u,
 		\end{equation*}
 		where, $\tau_n$ as in \eqref{tau}. Clearly, $v_n \in T_{u_n}\mathcal{M},\,\forall\,n\in\mathbb{N}$. Therefore, from \eqref{ps2} and using the uniform boundedness of $(u_n)$ due to $\left|\langle DI(u_n),u_n\rangle\right|=p|I(u_n)|$, we obtain
 		\begin{equation}\label{eq ps}
 			0\leq\lim_{n\to\infty}\left|\langle DI(u_n), u_n-u\rangle\right|\leq \lim_{n\to\infty}\left|\langle DI(u_n), v_n\rangle\right|+\lim_{n\to\infty}|1-\tau_n|\left|\langle DI(u_n), u_n\rangle\right|=0.
 		\end{equation}
 	 Since, $u_n\rightharpoonup u$, we have 
 	 \begin{equation}\label{eq ps1}
 	 	\lim_{n\to\infty}\Big|DI(u)[u_n-u]\Big|=0.
 	 \end{equation}
 	Therefore, from \eqref{eq ps} and \eqref{eq ps1}, we get
 		\begin{equation}\label{eq ps2}
 			\lim_{n\to\infty}\left|\langle DI(u_n)- DI(u), u_n-u\rangle\right|=0.
 		\end{equation}
 Therefore, proceeding similar to the steps as in the \textit{Proof of (vii)}, we conclude that $u_n\rightarrow u$ strongly in $\X$. Therefore, by the mountain pass theorem, there exists a critical point of $I$ attaining the critical value $\lambda_2(\Omega)$. This proves the claim.\\
 	It remains to show that $\lambda_2(\Omega)>\lambda_1(\Omega)$. We proceed with the method of contradiction. Let
 	\begin{equation*}
 		\lambda_2(\Omega)=\inf_{\phi\in \Gamma_1(\Omega)} \max_{u\in\mathrm{Im}(\phi)} \|u\|^p_{\X}=\lambda_1(\Omega)
 	\end{equation*}
 be true. Then, from the definition of $\lambda_2(\Omega)$, for each $n\in\mathbb{N}$, there exists $\phi_n\in\Gamma_1$ such that 
 		\begin{equation}\label{eq ev1}
 			\max_{u\in \phi_n(\mathbb{S}^1)} \|u\|^p_{\X} \leq \lambda_1(\Omega)+\frac{1}{n}.
 		\end{equation}
 	Let $e_1>0$ be the first eigenfunction to \eqref{sp eigen}. We know that $e_1>0$ or $e_1<0$ in $\Omega$.	Fix $\epsilon>0$, sufficiently small. Consider the following two disjoint neighberhoods of $e_1$
 	\begin{equation*}
 		B_{\epsilon}^+ = \{u \in \mathcal{M} : \|u-e_1\|_{L^p(\Omega)} < \epsilon\}
 		\qquad\text{and}\qquad
 		B_{\epsilon}^- = \{u\in \mathcal{M} : \|u- (-e_1)\|_{L^p(\Omega)}<\epsilon \}.
 	\end{equation*}
 	Note that $\phi_n(\mathbb{S}^{1})\not\subset B_{\epsilon}^+\cup B_{\epsilon}^-$ due to the fact $\phi_n\in\Gamma_1$ implying that $\phi_n(\mathbb{S}^{1})$ is symmetric and connected. Therefore, there exists $u_n\in \phi_n(\mathbb{S}^{1})\setminus \left(B_{\epsilon}^+\cup B_{\epsilon}^- \right)$ for each $n\in\mathbb{N}$. Moreover, the sequence $(u_n)$ is bounded in $\X$, thanks to \eqref{eq ev1}. Therefore, from Theorem \ref{l-3} there exists $v\in \mathcal{M}$ and a subsequence of $(u_n)$, (still denoted by $(u_n)$) such that $u_n\rightharpoonup v$ weakly in $\X$ and $u_n\rightarrow v$ strongly in $L^p(\Omega)$. By the choice of $u_n$, we have
 	\begin{equation*}
 		\|v\|^p_{\X}
 		\le \liminf_{n\to\infty} \|u_n\|^p_{\X}=\lambda_1(\Omega),
 	\end{equation*}
 	implying that $v\in \mathcal{M}$ is a global minimizer for $I$. Thus we get either $v=e_1$ or $v=-e_1$. Again, since $u_n\rightarrow v$ strongly in $L^p(\Omega)$, we have $v\in\mathcal{M}\setminus\left(B_{\epsilon}^+\cup B_{\epsilon}^- \right)$ giving us a contradiction. Hence, $\lambda_2(\Omega)>\lambda_1(\Omega)$. It is noteworthy to mention here that $\lambda_2(\Omega)$ admits only sign-changing eigenfunctions, thanks to (iii). In particular, using the isolatedness of $\lambda_1(\Omega)$, one can derive that if $\lambda>\lambda_1(\Omega)$ is an eigenvalue to \eqref{sp eigen}, then $\lambda\geq\lambda_2(\Omega)$.\\
    \underline{\textbf{{Proof of (ix):}}}
    Note that for $p=2$, the space $\X=X_0^{s,2}(\Omega)$ is a Hilbert space. From the definition of $\mathbb{P}_n$, it follows that each $\mathbb{P}_n$ is weakly closed and $\mathbb{P}_{n+1}\subset\mathbb{P}_n\subset\X$ for all $n\in\mathbb{N}$. This implies that
    \begin{equation}\label{2 lamb}
    	0<\lambda_{1} \leq \lambda_{2} \leq \ldots \leq \lambda_{n} \leq \lambda_{n+1} \leq \ldots.
    \end{equation}
    Thanks to the simplicity of $\lambda_1$, we conclude that $\lambda_1\neq\lambda_2$. Indeed, $\lambda_1=\lambda_2$ implies $e_2\in\mathbb{P}_2$ is an eigenfunction for $\lambda_1$. Thus $e_2=ce_1$ for some $c\neq0$. Now, $e_2\in\mathbb{P}_2$ gives that 
    $$0=\langle e_1,e_2\rangle=c\|e_1\|^2$$
    contradicts that $e_1\neq0$. thus, $\lambda_1\neq\lambda_2$.
    
    Again, we have
    \begin{equation}\label{eq 3.26}
    	\int_{\mathbb{G} \times \mathbb{G}}\frac{(e_{n+1}(x)-e_{n+1}(y))(\varphi(x)-\varphi(y))}{|y^{-1}x|^{Q+2s}} dxdy=\lambda_{n+1} \int_{\Omega} e_{n+1}(x) \varphi(x) dx, ~ \forall \varphi \in \mathbb{P}_{n+1}.
    \end{equation}

We need to show that \eqref{eq 3.26} holds for all $\varphi\in X_0^{s,2}(\Omega)$, to claim $\lambda_{n+1}$ is an eigenvalue. We use the principle of induction to establish our result. Clearly, $k=1$ is true since $\lambda_1$ is an eigenvalue. Assume, the result for $k=1,2,...,n$. We show it is true for $k=n+1$. Let us decompose the space $X_0^{s,2}(\Omega)$ as
	   \begin{equation}\label{eq 3.27}
	   	X_0^{s,2}(\Omega)=\operatorname{span}\left\{e_{1}, \ldots, e_{n}\right\} \oplus\left(\operatorname{span}\left\{e_{1}, \ldots, e_{n}\right\}\right)^{\perp}=\operatorname{span}\left\{e_{1}, \ldots, e_{n}\right\} \oplus \mathbb{P}_{n+1}.
	   \end{equation}
	    Note that $X_0^{s,2}(\Omega)$ being a Hilbert space, $\mathbb{P}_{n+1}$ as an orthogonal complement $\perp$ is well-defined due the inner product $\langle\cdot, \cdot\rangle$. Let $\varphi\in X_0^{s,2}(\Omega)$. Then there exist $\varphi_1\in span{\{e_{1}, \ldots, e_{n}\}}$ and $\varphi_2\in \mathbb{P}_{n+1}$ such that $\varphi=\varphi_1+\varphi_2$ and $\varphi_{1}=\sum_{i=1}^{n} c_{i} e_{i}$ for some scalars $c_1,c_2, \ldots, c_n$. Now, testing \eqref{eq 3.26} with $\varphi_{2}=\varphi-\varphi_{1}$, we derive that
	    
	    \begin{align}\label{eq 3.28}
	    	&\int_{\mathbb{G} \times \mathbb{G}}\frac{(e_{n+1}(x)-e_{n+1}(y))(\varphi(x)-\varphi(y))}{|y^{-1}x|^{Q+2s}} dxdy-\lambda_{n+1} \int_{\Omega} e_{n+1}(x) \varphi(x) dx\nonumber \\
	    	=&\int_{\mathbb{G} \times \mathbb{G}}\frac{(e_{n+1}(x)-e_{n+1}(y))(\varphi_1(x)-\varphi_1(y))}{|y^{-1}x|^{Q+2s}} dxdy
	    	-\lambda_{n+1} \int_{\Omega} e_{n+1}(x) \varphi_{1}(x) dx \nonumber\\
	    	=&\sum_{i=1}^{n} c_{i}\left[\int_{\mathbb{G} \times\mathbb{G}}\frac{\left(e_{n+1}(x)-e_{n+1}(y)\right)\left(e_{i}(x)-e_{i}(y)\right)}{|y^{-1}x|^{Q+2s}}dxdy-\lambda_{n+1} \int_{\Omega} e_{n+1}(x) e_{i}(x) dx\right].
	    \end{align}
	    Again, using the fact $\in \mathbb{P}_{n+1}$ and for $p=2$, testing \eqref{ev-soln} for each $e_{i}$, $i=1,2, \ldots, n$, we get
	    \begin{equation}\label{eq3.29}
	    	0=\int_{\mathbb{G} \times\mathbb{G}}\frac{\left(e_{n+1}(x)-e_{n+1}(y)\right)\left(e_{i}(x)-e_{i}(y)\right)}{|y^{-1}x|^{Q+2s}}dxdy=\lambda_{i} \int_{\Omega} e_{n+1}(x) e_{i}(x) d x
	    \end{equation}
	   Since, $\lambda_i\neq0,\,\forall\,i\in\mathbb{N}$, we deduce from \eqref{eq 3.28} that 
	    \begin{align}\label{eq 3.30}
	   	&\int_{\mathbb{G} \times \mathbb{G}}\frac{(e_{n+1}(x)-e_{n+1}(y))(\varphi(x)-\varphi(y))}{|y^{-1}x|^{Q+2s}} dxdy-\lambda_{n+1} \int_{\Omega} e_{n+1}(x) \varphi(x) dx=0,
	   \end{align}
	   for all $\varphi\in X_0^{s,2}(\Omega).$ Thus, $(\lambda_{n+1},e_{n+1})$ is an eigenpair of \eqref{pro1.3}. In fact, any eigenvalue of \eqref{pro1.3} is a member of the sequence $(\lambda_n)$ defined as in \eqref{eq3.10}. Indeed, let $\lambda\neq\lambda_n$ for all $n\in\mathbb{N}$ and $e$ be the corresponding eigenfunction such that $\|e\|_2=1$. Therefore, from \eqref{M}, we get
	   \begin{equation}\label{eq3.31}
	   	\lambda=\int_{\mathbb{G} \times \mathbb{G}}\frac{|e(x)-e(y)|^2}{|y^{-1}x|^{Q+2s}} dxdy=2I(e)\geq 2I(e_1)=\lambda_1,
	   \end{equation}
	   as $(\lambda_1,e_1)$ is the first eigenpair. Since, $\lambda\neq\lambda_n$ for all $n\in\mathbb{N}$ and $\lim_{n\rightarrow\infty}\lambda_n=\infty$, we conclude that $\lambda\in(\lambda_{n_0},\lambda_{n_0+1})$ for some $n_0\in\mathbb{N}$. Note that $e \notin \mathbb{P}_{n_0+1}$. If $e \in \mathbb{P}_{n_0+1}$, then from \eqref{eq3.11} and \eqref{eq3.31}, we get $\lambda=2I(e)\geq\lambda_{n_0+1}$ contradicts the fact that $\lambda\in(\lambda_{n_0},\lambda_{n_0+1})$. Thus, we have $e \notin \mathbb{P}_{n_0+1}$ and hence, we obtain $\langle e,e_i\rangle\neq0$ for some $i\in\{1,2,\ldots,n_0\}$. This contradicts that for the eigenpairs $(\lambda,e_{\lambda})$ and $(\mu,e_{\mu})$ with $\lambda\neq\mu$, we have $ \langle e_{\lambda},e_{\mu}\rangle=0$. Indeed, without loss of generality, we may assume that $e_{\lambda}\not\equiv0,\, e_{\mu}\not\equiv0$ such that $\|e_{\lambda}\|_2=1$ and $\|e_{\mu}\|=1$. Therefore, from \eqref{ev-soln} and \eqref{eq3.6}, we deduce that
	   \begin{align}
	   	\lambda\int_{\Omega} e_{\lambda}(x)e_{\mu}(x)dx&=\int_{\mathbb{G} \times \mathbb{G}}\frac{(e_{\lambda}(x)-e_{\lambda}(y))(e_{\mu}(x)-e_{\mu}(y))}{|y^{-1}x|^{Q+2s}} dxdy\nonumber\\
	   	&=\mu\int_{\Omega} e_{\lambda}(x)e_{\mu}(x)dx\nonumber\\
	   	\Rightarrow&(\lambda-\mu)\int_{\Omega} e_{\lambda}(x)e_{\mu}(x)dx=0,
	   \end{align}
	   proving that $\int_{\Omega} e_{\lambda}(x)e_{\mu}(x)dx=0$. Therefore, we get
	   \begin{align}\label{eq3.33}
	   	\langle e_{\lambda},e_{\mu}\rangle=&\int_{\mathbb{G} \times \mathbb{G}}\frac{(e_{\lambda}(x)-e_{\lambda}(y))(e_{\mu}(x)-e_{\mu}(y))}{|y^{-1}x|^{Q+2s}} dxdy=0.
	   \end{align}
	  Hence, we conclude that $\lambda\neq\lambda_n$ for all $n\in\mathbb{N}$ cannot be true. In other words, $\{\lambda_1,\lambda_2,\ldots,\lambda_n,\ldots\}$ is the set of all eigenvalues of \eqref{pro1.3} defined by \eqref{eq3.10}. This completes the proof.\\
	    \underline{\textbf{{Proof of (x):}}} Clearly, $\lambda_{n+1}$, the minimum of \eqref{eq3.10} is achieved at some $e_{n+1}\in\mathbb{P}_{n+1}$. Now, using \eqref{eq 3.30}, we get $e_{n+1}$ is an eigenfunction of \eqref{pro1.3}, completing the proof.\\
	    \underline{\textbf{{Proof of (xi):}}}
	We first prove the orthogonality of $(e_n)$. Let $i,j\in\mathbb{N}$ with $i> j$, then $i-1\geq j$. Thus $e_{i} \in \mathbb{P}_{i}=\left(\operatorname{span}\left\{e_1, e_2, \ldots, e_{i-1}\right\}\right)^{\perp} \subseteq\left(\operatorname{span}\left\{e_{j}\right\}\right)^{\perp}$ which implies 
	\begin{equation}\label{eq3.34}
		\left\langle e_i, e_j\right\rangle=0.
	\end{equation}
	Since, $e_i$ is an eigenfunction of \eqref{pro1.3}, then putting $u=e_i$ and $v=e_j$ in \eqref{ev-soln}, we get
	\begin{equation}\label{eq3.35}
		\int_{\mathbb{G} \times \mathbb{G}}\frac{(e_{\lambda}(x)-e_{\lambda}(y))(e_{\mu}(x)-e_{\mu}(y))}{|y^{-1}x|^{Q+2s}} dxdy=\lambda_i \int_{\Omega} e_i(x) e_j(x) dx.
	\end{equation}
	Therefore, we have	
	\begin{equation}\label{eq3.36}
		\langle e_i,e_j\rangle=0=\int_{\Omega} e_i(x)e_j(x)dx,
	\end{equation}
	which proves the orthogonality. Note that $\|e_n\|_2=1$. Therefore, it is sufficient to prove that $(e_n)$ is a basis for $X_0^{s,2}(\Omega)$ and $L^2(\Omega)$. The proof is standard but from the readers' point of view, we will provide the details. Let $v\in X_0^{s,2}(\Omega)\setminus\{0\}$ be such that
	\begin{equation}\label{eq3.37}
		\left\langle v, e_{n}\right\rangle=0~\text{for all $n \in \mathbb{N}$ and $\|v\|_2=1$.}
	\end{equation}  
	Therefore, using \eqref{eq3.10}, we get
	\begin{equation*}
		2I(v)<\lambda_{n_0+1}=\min _{\substack{u \in \mathbb{P}_{n_0+1} \\\|u\|_{L^{2}(\Omega)}=1}} \int_{\mathbb{G} \times \mathbb{G}}\frac{|e_{n_0}(x)-e_{n_0}(y)|^2}{|y^{-1}x|^{Q+2s}} dxdy
	\end{equation*}
	for some $n_0 \in \mathbb{N}$ which implies that $v \notin \mathbb{P}_{n_0+1}$. Therefore, there exists $m \in \mathbb{N}$ such that $\left\langle v, e_m\right\rangle \neq 0$, contradicting \eqref{eq3.37}. Hence, we conclude that 
	\begin{equation}\label{eq3.38}
		\left\langle v, e_{n}\right\rangle=0,~\forall\,n \in \mathbb{N} \text{ implies } v\equiv0. 
	\end{equation}
Now, define $v_n:=\frac{e_n }{\|e_n\|_{X_0^{s,2}(\Omega)}}$ and for any $g \in X_0^{s,2}(\Omega)$, $g_{n}:=\sum_{i=1}^{n}\left\langle g, v_i\right\rangle v_i$.
Thus, we have 	
	\begin{equation}\label{eq3.39}
		g_n\in \operatorname{span}\left\{e_{1},e_2, \ldots, e_{n}\right\},~\forall\,n \in \mathbb{N}.
	\end{equation}
Since, $(e_n)$ is orthogonal, then defining	$w_n:=g-g_n$, we obtain	
	\begin{align}
		0 &\leq \|w_n\|^{2}=\langle w_n, w_n\rangle\nonumber\\
		&=\|g\|^{2}+\|g_n\|^{2}-2\langle g,g_n\rangle\nonumber\\
		&=\|g\|^{2}+\langle g_n, g_n\rangle-2 \sum_{i=1}^{n}\langle g, v_i\rangle^2\nonumber\\
		& =\|g\|^{2}-\sum_{i=1}^{n}\langle g, v_i\rangle^2\nonumber\\
		&\Rightarrow \sum_{i=1}^{n}\langle g, v_i\rangle^2\leq \|g\|^{2}~\text{ for all }n\in \mathbb{N}.
	\end{align}
This implies that the series $\sum_{i=1}^{\infty}\langle g, v_i\rangle^{2}$ is convergent. Thus the sequence of partial sums is Cauchy in $\mathbb{R}$. Let $t_n=\sum_{i=1}^{n}\langle g, v_i\rangle^{2}$. Then using the orthogonality of $(e_n)$ in $X_0^{s,2}(\Omega)$, we have	
	\begin{align*}
		\|w_m-w_n\|^{2}&=\left\|\sum_{i=n+1}^{m}\langle g, v_i\rangle v_i \right\|^{2} \\
		& =\sum_{i=n+1}^{m}\langle g, v_i\rangle^{2}\\
		&=t_m-t_n,~\text{ for all } m>n.
	\end{align*}
This asserts that $(w_n)$ is Cauchy sequence in the Hilbert space $X_0^{s,2}(\Omega)$. Thus, there exists $w\in X_0^{s,2}(\Omega)$ such that $w_n\rightarrow w$ as $n\rightarrow\infty$. Note that
\begin{equation}
	\langle w_m, v_n\rangle=\langle g, v_n\rangle-\langle g_m, v_n\rangle=\langle g, v_n\rangle-\langle g, v_n\rangle=0.
\end{equation}
	Therefore, we get $\langle w, v_n\rangle=\lim_{n\rightarrow\infty}\langle w_m, v_n\rangle=0,~\text{for all } n\in\mathbb{N}.$ Hence, from \eqref{eq3.38}, we have $w\equiv0$. therefore, we have
	\begin{equation}\label{eq3.42}
		\lim_{n\rightarrow\infty}g_n=g-w_n =g-w=g \in X_0^{s,2}(\Omega)
	\end{equation}
Thus from \eqref{eq3.39} and \eqref{eq3.42}, we get $(e_n)$ forms an orthogonal basis for $X_0^{s,2}(\Omega)$. Finally, we show that $(e_n)$ forms a basis for $L^2(\Omega)$. Indeed, for $\eta\in L^2(\Omega)$ choose $\eta_m\in C_c^{\infty}(\Omega)$ such that $\|\eta-\eta_m\|<\frac{1}{m}$. Obviously, $\eta_m\in X_0^{s,2}(\Omega)$. Therefore, there exists $m_n\in\mathbb{N}$ and $w_m\in span\{e_1, e_2,\ldots,e_{n_m}\}$ such that $\|\eta_m-w_m\|<\frac{1}{m}$, since, $(e_n)$ forms a basis for $X_0^{s,2}(\Omega)$. Thus, by embedding result, we get
\begin{equation*}
	\|\eta_m-w_m\|_{L^{2}(\Omega)}\leq C\|\eta_m-w_m\|_{X_0^{s,2}(\Omega)}\leq \frac{C}{m}
\end{equation*}
for some $C>0$ and from the triangle inequality, we obtain
\begin{equation*}
	\|\eta-w_m\|_{L^{2}(\Omega)} \leq \|\eta-\eta_m\|_{L^{2}(\Omega)}+\|\eta_m-w_m\|_{L^{2}(\Omega)} \leq \frac{C+1}{m},
\end{equation*}
	proving that $(e_n)$ forms a basis for $L^{2}(\Omega)$.\\
	\underline{\textbf{{Proof of (xii):}}} Suppose there exists $m \in \mathbb{N}\cup\{0\},$ such that
	\begin{equation}\label{eq3.43}
		\lambda_{n-1} < \lambda_n = \cdot =\lambda_{n+m}<\lambda_{n+m+1}.
	\end{equation}
	Therefore, from the previous part, we have every element of $\operatorname{span}\left\{e_n, \ldots, e_{n+m}\right\}$ is an eigenfunction of \eqref{pro1.3} for $\lambda_n=\ldots,=\lambda_{n+m}$. We again decompose the space $X_0^{s,2}(\Omega)$ as
	$$X_0^{s,2}(\Omega)=\operatorname{span}\left\{e_n, \ldots, e_{n+m}\right\} \oplus\left(\operatorname{span}\left\{e_n, \ldots, e_{n+m}\right\}\right)^{\perp}.	$$
	Let $\xi \not \equiv 0$ be an eigenfunction corresponding to $\lambda_{n}\in \operatorname{span}\left\{e_n, \ldots, e_{n+m}\right\}$. Therefore, there exist $\xi_1, \xi_2$ such that
	\begin{equation}\label{eq3.44}
		\xi=\xi_{1}+\xi_{2} \text{ with } \left\langle\xi_{1}, \xi_{2}\right\rangle=0,
	\end{equation} 
	where
\begin{equation}\label{eq3.45}
	\xi_{1} \in \operatorname{span}\left\{e_n, \ldots, e_{n+m}\right\} \text { and } \xi_{2} \in\left(\operatorname{span}\left\{e_n, \ldots, e_{n+m}\right\}\right)^{\perp}
\end{equation}
On testing \eqref{ev-soln} with $\xi$ as test function and using \eqref{eq3.44} we obtain
\begin{align}\label{eq3.46}
	\lambda_{n}\|\xi\|_{L^{2}(\Omega)}^{2}&=\int_{\mathbb{G} \times \mathbb{G}}\frac{|\xi(x)-\xi(y)|^2}{|y^{-1}x|^{Q+2s}} dxdy=\|\xi\|^{2}\nonumber\\
	&=\|\xi_{1}\|^{2}+\|\xi_{2}\|^2.
\end{align}
Note that $e_n, \ldots, e_{n+m}$ are eigenfunctions corresponding to $\lambda_{n}=\cdots=\lambda_{n+m}$. On using \eqref{eq3.45}, we have $(\lambda_{k}, \xi_1)$ is also an eigenpair. Again testing \eqref{ev-soln} with $\xi_{2}$ and using \eqref{eq3.44}, we derive that 
\begin{equation}\label{eq3.47}
	\int_{\Omega} \xi_{1}(x) \xi_{2}(x) dx=0, 
\end{equation}
since
\begin{align}\label{eq3.48}
	\lambda_{n} \int_{\Omega} \xi_{1}(x) \xi_{2}(x) dx&=\int_{\mathbb{G} \times \mathbb{G}}\frac{(\xi_1(x)-\xi_1(y))(\xi_2(x)-\xi_2(y))}{|y^{-1}x|^{Q+2s}} dxdy =\langle\xi_{1}, \xi_{2}\rangle=0.
\end{align}
Therefore, from \eqref{eq3.47}, we have
\begin{equation}\label{eq3.49}
	\|\xi\|_{L^{2}(\Omega)}^{2}=\left\|\xi_{1}+\xi_{2}\right\|_{L^{2}(\Omega)}^{2}=\left\|\xi_{1}\right\|_{L^{2}(\Omega)}^{2}+\left\|\xi_{2}\right\|_{L^{2}(\Omega)}^{2}.
\end{equation}
Since, $\xi_{1} \in \operatorname{span}\left\{e_n, \ldots, e_{n+m}\right\}$, there exist scalars $c_i,\,i=n,\ldots,n+m$ such that
\begin{equation}\label{eq3.50}
	\xi_{1}=\sum_{i=n}^{n+m} c_{i} e_{i}.
\end{equation}
Therefore, using \eqref{eq x} and the orthogonality, we get
\begin{align}\label{eq3.51}
	\|\xi_{1}\|^{2}&=\sum_{i=n}^{n+m} c_{i}^{2}\|e_{i}\|^{2}=\sum_{i=n}^{n+m} c_{i}^{2} \lambda_{i}
	=\lambda_{n} \sum_{i=n}^{n+m} c_{i}^{2}=\lambda_{n}\|\xi_{1}\|_{L^{2}(\Omega)}^{2}.
\end{align}
Now, $\xi, \xi_1$ are eigenfunctions imply that $\xi_2$ is also an eigenfunction for $\lambda_n$.
 Therefore, from \eqref{eq3.33} and \eqref{eq3.43}, we deduce
\begin{equation}\label{eq3.52}
	\langle\xi_{2}, e_{1}\rangle=\cdots=\langle\xi_{2}, e_{n-1}\rangle=0.
\end{equation}
Thus, we use \eqref{eq3.45} and \eqref{eq3.52} to obtain
\begin{equation}\label{eq3.53}
	\xi_{2} \in\left(\operatorname{span}\left\{e_{1}, \ldots, e_{n+m}\right\}\right)^{\perp}=\mathbb{P}_{n+m+1}.
\end{equation}
We conclude the proof by showing $\xi=\xi_1$, that is $\xi_2\equiv0$. Suppose, $\xi\neq0$. Then, using  \eqref{eq3.10-a} and \eqref{eq3.53}, we obtain
\begin{align}\label{eq3.54}
	\lambda_n< \lambda_{n+m+1}&= \min_{{u \in \mathbb{P}_{n+m+1}}\setminus\{0\}} \frac{ \int_{\mathbb{G} \times \mathbb{G}}  \frac{|u(x)-u(y)|^2}{|y^{-1}x|^{Q+2s}}dxdy}{\int_{\Omega} |u(x)|^2\,dx} \nonumber\\
	& \leq \frac{ \int_{\mathbb{G} \times \mathbb{G}}  \frac{|u(x)-u(y)|^2}{|y^{-1}x|^{Q+2s}}dxdy}{\int_{\Omega} |u(x)|^2\,dx} \nonumber\\
	&=\frac{\|\xi_{2}\|^{2}}{\|\xi_{2}\|_{L^{2}(\Omega)}^{2}}.
\end{align}
Therefore, from \eqref{eq3.46}, \eqref{eq3.49}, \eqref{eq3.51} and \eqref{eq3.54}, we deduce that
\begin{align}\label{eq3.55}
	\lambda_n\|\xi\|_{L^{2}(\Omega)}^{2}&=\|\xi_{1}\|^{2}+\|\xi_{2}\|^{2} \nonumber\\
	&>\lambda_n\|\xi_{1}\|_{L^{2}(\Omega)}^{2}+\lambda_{n}\|\xi_{2}\|_{L^{2}(\Omega)}^{2}\nonumber \\
	&=\lambda_n\|\xi\|_{L^{2}(\Omega)}^{2}
\end{align}
Thus, we arrive at a contradiction. Therefore, $\xi_2\equiv0$. Hence, from \eqref{eq3.43}, we conclude that
\begin{equation*}
	\xi=\xi_{1} \in \operatorname{span}\left\{e_n, \ldots, e_{n+m}\right\},
\end{equation*}
completing the proof.
\end{proof}

\section{Bifurcation and multiplicity of solutions} \label{s4}
In this section, we prove one of the main results (Theorem \ref{mainthm}) of this paper concerning the bifurcation and multiplicity. We recall from the following abstract multiplicity result for critical points of an even, $C^1$ functional.  

\begin{theorem} \label{critithm}
     Let $\mathcal{H}$ be a real Hilbert space with norm $\|\cdot\|$ and $ \mathfrak{I} \in C^{1}(\mathcal{H}, \mathbb{R})$ be a functional on $\mathcal{H}$ satisfying the following assumptions:
\begin{itemize}
    \item[(a)] $\mathfrak{I}(0)=0, \mathfrak{I}(u)=\mathfrak{I}(-u) \quad$ for any $u \in \mathcal{H}$;
    \item[(b)] there exists $\beta>0$ such that the Palais-Smale condition for $\mathfrak{I}$ holds in $( 0, \beta)$;
    \item[(c)] there exist two closed subspaces $\mathcal{V}, \mathcal{W} \subset \mathcal{H}$ and positive constants $\rho, \delta, \beta^{\prime}$, with $\delta<\beta^{\prime}<\beta$ such that
    \begin{itemize}
        \item[(i)] $\mathfrak{I}(u) \leq \beta^{\prime}$ for any $u \in \mathcal{W},$
        \item[(ii)]  $\mathfrak{I}(u) \geq \delta$ for any $u \in \mathcal{V},\|u\|=\rho,$
        \item[(iii)] codim $\mathcal{V}<+\infty$ and $\operatorname{dim} \mathcal{W} \geq \operatorname{codim} \mathcal{V}$.
    \end{itemize}
\end{itemize}
 Then, there exists at least $
\operatorname{dim} \mathcal{W}-\operatorname{codim} \mathcal{V}
$
pairs of critical points of the functional $\mathfrak{I}$ with critical values belonging to the interval $\left[\delta, \beta^{\prime}\right]$.

\end{theorem}

We present the proof of Theorem \ref{mainthm} by applying the above abstract critical point theorem to the Euler-Lagrange functional associated with  the problem \eqref{pro1intro} is given by 
\begin{equation} \label{func1}
    \mathfrak{I}_{\lambda}(u):= \frac{1}{2} \int_{\G \times \G} \frac{|u(x)-u(y)|^2}{|y^{-1}x|^{Q+2s}} dx dy - \frac{1}{2_s^*} \int_{\Omega} |u(x)|^{2_s^*} dx - \frac{\lambda}{2} \int_{\Omega} |u(x)|^2\, dx.
\end{equation}
It is clear that $\mathfrak{I}_{\lambda}$ is of class $C^1(X_0^{s,2}(\Omega))$ and 
\begin{align} \label{derfuc1}
    \langle \mathfrak{I}_{\lambda}'(u), \phi \rangle &=  \int_{\G \times \G} \frac{(u(x)-u(y))(\phi(x)-\phi(y))}{|y^{-1}x|^{Q+2s}} dx dy -  \int_{\Omega} |u(x)|^{2_s^*-2} u(x) \phi(x)dx \\&\quad\quad\quad - \lambda \int_{\Omega} |u(x)| \phi(x)\, dx \nonumber
\end{align}
for any $u, \phi \in X_0^{s,2}(\Omega).$ It is evident from here that the critical points of the energy functional $\mathfrak{I}_\lambda$ are solutions to Problem \eqref{pro1intro}.

To apply Theorem \ref{critithm}, we need to verify all the assumptions of it. Noting that the functional $\mathfrak{I}_\lambda$ is even and $\mathfrak{I}_\lambda(0)=0,$ we conclude that assumption $(a)$ of Theorem \ref{critithm} is verified. 

Next for the assumption $(b),$ we need to show that the functional $\mathfrak{I}_\lambda$ satisfies the Palais-Smale (P-S) condition at some level $c \in \mathbb{R},$ that is, every sequence $(u_j)_{j \in \mathbb{N}}$ in $X_0^{s,2}(\Omega)$ such that \begin{equation}\label{ps1}
    \mathfrak{I}_\lambda(u_j) \rightarrow c \quad \text{as} \quad j \rightarrow \infty
\end{equation} and 
\begin{equation} \label{PS2}
    \sup \{|\langle \mathfrak{I}_\lambda'(u_j), \varphi \rangle|: \varphi \in X^s_0(\Omega), \|\varphi\|_{X^s_0(\Omega)}=1\} \rightarrow 0
\end{equation}
as $j \rightarrow \infty,$ admits a subsequence strongly convergent in $X_0^{s,2}(\Omega).$ 

We prove the following lemma for this purpose. First, let us denote
\begin{align}\label{bestcon}
    C_{2^*_s, \Om}:= \inf_{u \in X_0^{s,p}(\Om) \backslash \{0\}} \frac{\Bigg( \int_{\G \times \G} \frac{|u(x)-u(y)|^2}{|y^{-1} x|^{Q+2s}} d x d y \Bigg)^{1/2}}{\|u\|_{L^{2_s^*(\Om)}} }
\end{align}
as the best constant of the Sobolev embedding $ X_0^{s,2}(\Omega) \hookrightarrow L^{2_s^*}(\Omega)$ given by \eqref{emG}.
\begin{lemma} \label{lemm4.1}
    For any value $c$ such that
    \begin{equation} \label{psss}
        c< \frac{s}{Q} C_{2_s^*, \Omega}^{Q/s} 
    \end{equation}
    the functional $\mathfrak{I}_\lambda$ satisfies the Palais-Smale condition at the level $c.$
\end{lemma}
\begin{proof}
    Let $(u_j)_{j \in \N}$ be a sequence satisfying the Palais-Smale conditions. We first claim that the sequence $(u_j)$ is bounded in $X_0^{s,2}(\Omega).$ Indeed, for any $j \in \N,$ it follows from \eqref{ps1} and \eqref{PS2} that  there exists $k>0$ such that 
    \begin{equation} \label{eq3.4}
        |\mathfrak{I}_\lambda(u_j)| \leq k
    \end{equation}
    and 
    $$ \left|\langle \mathfrak{I}_\lambda'(u_j), \frac{u_j}{\|u_j\|_{X^s_0(\Omega)}} \rangle \right| \leq k.$$
    Therefore, we have 
    \begin{equation} \label{2.7eq}
        \mathfrak{I}_\lambda(u_j)- \frac{1}{2} \langle \mathfrak{I}_\lambda'(u_j), u_j \rangle \leq k (1+\|u_j\|_{X^s_0(\Omega)}).
    \end{equation}
    Furthermore, we deduce from \eqref{func1} and \eqref{derfuc1} that 
    $$ \mathfrak{I}_\lambda(u_j)- \frac{1}{2} \langle \mathfrak{I}_\lambda'(u_j), u_j \rangle = \left( \frac{1}{2}-\frac{1}{2_s^*} \right) \|u_j\|_{L^{2_s^*}(\Omega)}^{2_s^*} = \frac{s}{Q} \|u_j\|_{L^{2_s^*}(\Omega)}^{2_s^*}. $$
    Thus, from \eqref{2.7eq} it follows that, for any $j \in \N$, we have
    \begin{equation} \label{eqq3.6}
        \|u_j\|_{L^{2_s^*}(\Omega)}^{2_s^*} \leq \tilde{k} (1+\|u_j\|_{X^s_0(\Omega)})
    \end{equation}
    for some suitable positive constant $\tilde{k}.$
    Now, we apply the H\"older inequality to get
    $$ \|u_j\|^2_{L^2(\Omega)} \leq |\Omega|^{2s/Q} \|u_j\|_{L^{2_s^*}(\Omega)}^2 \leq \tilde{k}^{2/2_s^*} |\Omega|^{2s/Q}(1+\|u_j\|_{X^s_0(\Omega)})^{2/2_s^*},$$
    and then use the fact $2_s^*>2$ to deduce that
    \begin{equation} \label{eq3.7}
        \|u_j\|^2_{L^2(\Omega)} \leq k_* (1+\|u_j\|_{X^s_0(\Omega)}),
    \end{equation} for some $k_*>0$, independent of $j.$
    Finally, putting together \eqref{eq3.4}, \eqref{eqq3.6} and \eqref{eq3.7}, we get
    \begin{align*}
        k &\geq \mathfrak{I}_\lambda(u_j)= \frac{1}{2} \int_{\G \times \G} \frac{|u_j(x)-u_j(y)|^2}{|y^{-1}x|^{Q+2s}} dx dy - \frac{1}{2_s^*} \int_{\Omega} |u_j(x)|^{2_s^*} dx - \frac{\lambda}{2} \int_{\Omega} |u_j(x)|^2\, dx \\& \geq \frac{1}{2} \|u_j\|_{X^s_0(\Omega)}^2- \bar{k} (1+\|u_j\|_{X^s_0(\Omega)}),
    \end{align*}
    for some $\bar{k}>0$, independent of $j,$ showing that $(u_j)$ is bounded a sequence in $X_0^{s,2}(\Omega).$

    Now, we will show that there exists $u_\infty \in X_0^{s,2}(\Omega)$ which is a solution of \eqref{pro1intro}. 
    Since $(u_j)$ is bounded in the Hilbert space $X_0^{s,2}(\Omega)$, there exists a subsequence, again denoted by $(u_j),$ such that $u_j \rightarrow u_\infty$ weakly for some $u_\infty \in X^s_0(\Omega)$, that is, 
    \begin{align} \label{eqq39}
        \int_{\G \times \G} \frac{(u_j(x)-u_j(y))(\phi(x)-\phi(y))}{|y^{-1}x|^{Q+2s}}dx dy \rightarrow \int_{\G \times \G} \frac{(u_\infty(x)-u_\infty(y))(\phi(x)-\phi(y))}{|y^{-1}x|^{Q+2s}}dx dy
    \end{align}for any $\phi \in X^s_0(\Omega),$ as $j \rightarrow \infty.$
 
 Further, using the embedding result (Theorem \ref{l-3}) of $X_0^{s,2}(\Omega)$ into the Lebesgue spaces and the fact that $L^{2_s^*}(\Omega)$ is a reflexive space along with the boundedness of $(u_j)$ in $X_0^{s,2}(\Omega)$ and \eqref{eqq3.6} we obtain, up to a subsequence, that  
\begin{align}
    u_j &\rightarrow u_\infty\quad \text{weakly in}\,\,\, L^{2_s^*}(\Omega)\label{eq310}\\
     u_j &\rightarrow u_\infty \,\,\,\text{in}\,\,L^2(\Omega)\label{eqq31.1}~\text{and}\\
     u_j &\rightarrow u_\infty \quad \text{a.e in} \,\,\, \Omega\label{eq310'}~\text{as}~j \rightarrow \infty.
\end{align}

Since $(u_j)$ is a bounded in $X_0^{s,2}(\Omega)$ and \eqref{eqq3.6} holds, we can conclude that $\|u_j\|_{L^{2^*_s}(\Omega)}$ is bounded uniformly in $j.$ This implies that the sequnece $(|u_j|^{2_s^*-2}u_j)$ is bounded in $L^{2_s^*/ 2_s^*-1}(\Omega)$ uniformly in $j.$ 
Therefore, from \eqref{eq310} we obtain
\begin{equation}\label{eqqpn}
    |u_j|^{2_s^*-2}u_j \rightarrow |u_\infty|^{2_s^*-2}u_\infty\quad \text{weakly in} \,\,L^{2_s^*/ 2_s^*-1}(\Omega)~\text{as}~j \rightarrow \infty.
\end{equation}
Thus, for all $\phi \in L^{2_s^*}(\Omega)=(L^{2_s^* / 2_s^*-1}(\Omega))^{'}$ we have that 
\begin{equation}
    \int_{\Omega} |u_j|^{2_s^*-2}u_j \phi(x) dx \rightarrow \int_{\Omega} |u_\infty|^{2_s^*-2}u_\infty \phi(x) dx~\text{as}~j \rightarrow \infty.
\end{equation}
In particularly, as $X_0^{s,2}(\Omega) \subseteq L^{2_s^*}(\Omega)$ gives that 
\begin{equation} \label{eqq3.14}
    \int_{\Omega} |u_j|^{2_s^*-2}u_j \phi(x) dx \rightarrow \int_{\Omega} |u_\infty|^{2_s^*-2}u_\infty \phi(x) dx, ~\text{for every}~\phi \in X_0^{s,2}(\Omega).
\end{equation}

Recalling that $(u_j)_{j}$ is a Palais-Smale sequence, for any $\phi \in X_{0}^s(\Omega)$
 \begin{align*}
    o(1)&=  \langle \mathfrak{I}_{\lambda}'(u), \phi \rangle \\&=  \int_{\G \times \G} \frac{(u(x)-u(y))(\phi(x)-\phi(y))}{|y^{-1}x|^{Q+2s}} dx dy -  \int_{\Omega} |u(x)|^{2_s^*-2} u(x) \phi(x)dx - \lambda \int_{\Omega} |u(x)| \phi(x)\, dx
 \end{align*}
 as $j \rightarrow \infty.$ Now, letting $j \rightarrow \infty$ and applying \eqref{eqq39}, \eqref{eqq31.1} and \eqref{eqq3.14} we find that 
 \begin{align} \label{eqqw}
  \nonumber   & \int_{\G \times \G} \frac{(u_\infty(x)-u_\infty(y))(\phi(x)-\phi(y))}{|y^{-1}x|^{Q+2s}} dx dy -  \int_{\Omega} |u_\infty(x)|^{2_s^*-2} u_\infty(x) \phi(x)dx \\&\quad - \lambda \int_{\Omega} |u_\infty(x)| \phi(x)\, dx=0.
 \end{align}
 This implies that $u_\infty \in X_0^{s,2}(\Omega)$ is a weak solution of problem \eqref{pro1intro}. We now show that 
\begin{align}
    u_j \rightarrow u_\infty \quad \text{in} \quad X_0^{s,2}(\Omega).
\end{align}
Let us deduce the following three important relations for the function $u_\infty \in X_0^{s,2}(\Omega).$ 
\begin{align} \label{fim}
    \mathfrak{I}_\lambda (u_\infty)= \frac{s}{Q} \int_{\Omega} |u_\infty|^{2_s^*} dx \geq 0. 
\end{align}
Indeed, taking $u_\infty \in X_0^{s,2}(\Omega)$ as a test function in \eqref{eqqw} we have
$$\int_{\G \times \G} \frac{|u_\infty(x)-u_\infty(y)|^2}{|y^{-1}x|^{Q+2s}} dx dy  - \lambda \int_{\Omega} |u_\infty(x)|^2  dx=\int_{\Omega} |u_\infty(x)|^{2_s^*} dx ,$$
which further implies that 
$$\mathfrak{I}_\lambda(u_\infty)= \left( \frac{1}{2}-\frac{1}{2^*} \right) \int_{\Omega} |u_\infty(x)|^{2_s^*}\,dx = \frac{s}{Q} \int_{\Omega} |u_\infty(x)|^{2_s^*}\,dx \geq 0,$$
proving \eqref{fim}. The second important relation is 
\begin{align}\label{sim}
 \nonumber   \mathfrak{I}_{\lambda}(u_j)&=\mathfrak{I}_{\lambda}(u_\infty)-\frac{1}{2^*_s} \int_{\Omega} |u_j(x)-u_\infty(x)|^{2_s^*}\,dx \\&+ \frac{1}{2}\int_{\G \times \G} \frac{|u_j(x)-u_\infty(x)-u_j(y)+u_\infty(y)|^2}{|y^{-1}x|^{Q+2s}} dx dy+o(1).
\end{align}
To see this, we apply \eqref{eq310} and \eqref{eqq3.14} along with Brezis-Lieb Lemma \cite[Theorem 1]{BL83} and Theorem \ref{l-3} to get, as $j \rightarrow \infty,$
\begin{align} \label{pn}
   \nonumber \int_{\G \times \G} \frac{|u_j(x)-u_j(y)|^2}{|y^{-1}x|^{Q+2s}} dx dy=& \int_{\G \times \G} \frac{|u_j(x)-u_\infty(x)-u_j(y)+u_\infty(y)|^2}{|y^{-1}x|^{Q+2s}} dx dy \\&+\int_{\G \times \G} \frac{|u_\infty(x)-u_\infty(y)|^2}{|y^{-1}x|^{Q+2s}} dx dy+o(1)
\end{align}
and 
\begin{align} \label{pnn}
    \int_{\Omega} |u_j(x)|^{2_s^*}\,dx = \int_{\Omega} |u_j(x)-u_\infty(x)|^{2_s^*}\,dx+ \int_{\Omega} |u_\infty(x)|^{2_s^*}\,dx+o(1).
\end{align}
The above two equations \eqref{pn} and \eqref{pnn} combined with \eqref{eqq31.1} yield that
\begin{align*}
    \mathfrak{I}_\lambda(u_j)&= \frac{1}{2} \int_{\G \times \G} \frac{|u_j(x)-u_\infty(x)-u_j(y)+u_\infty(y)|^2}{|y^{-1}x|^{Q+2s}} dx dy+\frac{1}{2} \int_{\G \times \G} \frac{|u_\infty(x)-u_\infty(y)|^2}{|y^{-1}x|^{Q+2s}} dx dy\\&- \frac{1}{2_s^*}  \int_{\Omega} |u_j(x)-u_\infty(x)|^{2_s^*}\,dx-\frac{1}{2_s^*} \int_{\Omega} |u_\infty(x)|^{2_s^*}\,dx-\frac{\lambda}{2} \int_{\Omega} |u_\infty(x)|^2\,dx+o(1) \\&
    =\mathfrak{I}_{\lambda}(u_\infty)- \frac{1}{2_s^*}  \int_{\Omega} |u_j(x)-u_\infty(x)|^{2_s^*}\,dx\\&\quad\quad+\frac{1}{2}  \int_{\G \times \G} \frac{|u_j(x)-u_\infty(x)-u_j(y)+u_\infty(y)|^2}{|y^{-1}x|^{Q+2s}} dx dy+o(1)
\end{align*}
as $j\rightarrow \infty$ proving the related as desired. The final relation is, as $j \rightarrow \infty,$
\begin{align}\label{tim}
    \int_{\G \times \G} \frac{|u_j(x)-u_\infty(x)-u_j(y)+u_\infty(y)|^2}{|y^{-1}x|^{Q+2s}} dx dy =\int_{\Omega} |u_j(x)-u_\infty(x)|^{2_s^*}\,dx+o(1).
\end{align}
To prove this, we use \eqref{eq310}, \eqref{eqqpn} and \eqref{pnn} to first obtain that, as $j \rightarrow \infty,$
\begin{align}\label{A3}
 \nonumber   \int_{\Omega} &\left( |u_j(x)|^{2_s^*-2} u_j(x)- |u_\infty(x)|^{2_s^*-2} u_\infty(x)\right) (u_j(x)-u_\infty(x))\,dx \\=& \int_{\Omega} |u_j(x)|^{2_s^*}\,dx-\int_{\Omega} |u_\infty|^{2_s^*-2} u_\infty(x) u_j(x)\,dx-\int_{\Omega} |u_j|^{2_s^*-2} u_j(x) u_\infty(x)\,dx\nonumber \\&+\int_{\Omega} |u_\infty(x)|^{2_s^*}\,dx \nonumber
 \\&=\nonumber \int_{\Omega} |u_j(x)|^{2_s^*}\,dx -\int_{\Omega} |u_\infty(x)|^{2_s^*}dx+o(1) \\& = \int_{\Omega} |u_j(x)-u_\infty(x)|^{2_s^*}\,dx+o(1).
\end{align}
Next, using the fact $(u_j)$ is a PS sequence (particularly \eqref{PS2}) and $(u_j)$ is a bounded sequence with the fact that $u_\infty$ is a solution of \eqref{pro1intro}, we have 
\begin{equation} \label{eqq325}
    o(1)=\langle \mathfrak{I}_\lambda'(u_j), u_j-u_\infty  \rangle =\langle \mathfrak{I}_\lambda'(u_j)-\mathfrak{I}_\lambda'(u_\infty), u_j-u_\infty  \rangle
\end{equation}
as $j \rightarrow \infty.$
On the other hand, by \eqref{A3} and \eqref{eqq31.1}, it follows that
\begin{align} \label{eqq326}
  \nonumber  \langle \mathfrak{I}_\lambda^{'}(u_j)-\mathfrak{I}_\lambda^{'}(u_\infty), u_j-u_\infty  \rangle &= \int_{\G \times \G} \frac{|u_j(x)-u_\infty(x)-u_j(y)+u_\infty(y)|^2}{|y^{-1}x|^{Q+2s}} dx dy  \\&\nonumber -\lambda  \int_{\Omega} |u_j(x)-u_\infty(x)|^2\,dx \\& \nonumber -\int_{\Omega} \left( |u_j(x)|^{2_s^*-2} u_j(x)- |u_\infty(x)|^{2_s^*-2} u_\infty(x)\right) (u_j(x)-u_\infty(x))\,dx \\&\nonumber =\int_{\G \times \G} \frac{|u_j(x)-u_\infty(x)-u_j(y)+u_\infty(y)|^2}{|y^{-1}x|^{Q+2s}} dx dy\\&- \int_{\Omega} |u_j(x)-u_\infty(x)|^{2_s^*}\,dx+o(1)
\end{align} as $j \rightarrow \infty.$
Finally, \eqref{eqq325} together with \eqref{eqq326} provide the required relation \eqref{tim}. Combining \eqref{sim} and \eqref{tim} we get
\begin{align*}
     \mathfrak{I}_{\lambda}(u_j)&=\mathfrak{I}_{\lambda}(u_\infty)-\frac{1}{2^*_s} \int_{\Omega} |u_j(x)-u_\infty(x)|^{2_s^*}\,dx \\&+ \frac{1}{2}\int_{\G \times \G} \frac{|u_j(x)-u_\infty(x)-u_j(y)+u_\infty(y)|^2}{|y^{-1}x|^{Q+2s}} dx dy+o(1) \\&= \mathfrak{I}_{\lambda}(u_\infty) +\left(\frac{1}{2}-\frac{1}{2_s^*} \right) \int_{\G \times \G} \frac{|u_j(x)-u_\infty(x)-u_j(y)+u_\infty(y)|^2}{|y^{-1}x|^{Q+2s}} dx dy+o(1)
     \\&= \mathfrak{I}_{\lambda}(u_\infty) +\frac{s}{Q} \int_{\G \times \G} \frac{|u_j(x)-u_\infty(x)-u_j(y)+u_\infty(y)|^2}{|y^{-1}x|^{Q+2s}} dx dy+o(1).
\end{align*}
This together with \eqref{ps1} gives, as $j \rightarrow \infty,$ that
\begin{align} \label{e327}
    c=\mathfrak{I}_{\lambda}(u_\infty) +\frac{s}{Q} \int_{\G \times \G} \frac{|u_j(x)-u_\infty(x)-u_j(y)+u_\infty(y)|^2}{|y^{-1}x|^{Q+2s}} dx dy+o(1).
\end{align}
As the sequence $(\|u_j\|_{X_0^{s,2}(\Omega)})$ is bounded in $\mathbb{R}$ by the fact that the sequence $(u_j)$ is bounded in $X_0^{s,2}(\Omega),$ we get 
\begin{equation}\label{eq328}
    \|u_j-u_\infty\|_{X_0^{s,2}(\Omega)}^2=\int_{\G \times \G} \frac{|u_j(x)-u_\infty(x)-u_j(y)+u_\infty(y)|^2}{|y^{-1}x|^{Q+2s}} dx dy \rightarrow S
\end{equation}
for some $S \in [0, \infty)$ as $j \rightarrow \infty,$ up to a subsequence, if necessary. Thus, \eqref{e327}  gives 
\begin{equation} \label{eqq329}
    c=\mathfrak{I}_{\lambda}(u_\infty) +\frac{s}{Q} S \geq \frac{s}{Q} S,
\end{equation}
as an application of \eqref{fim} and \eqref{eq328}. On the other hand, \eqref{tim}, \eqref{eq328} and \eqref{bestcon}  provide that 
\begin{equation}
   S^{\frac{1}{2_s^*}} C_{2_s^*, \Om} \leq S^{\frac{1}{2}} 
\end{equation} implying that
either $S=0$ or $C_{2_s^*, \Om}^{\frac{Q}{s}} \leq S.$ Suppose that $C_{2_s^*, \Om}^{\frac{Q}{s}} \leq S,$ then \eqref{eqq329} yields that
 \begin{equation*}
     c \geq \frac{s}{Q} S \geq \frac{s}{Q} C_{2_s^*, \Om}^{\frac{Q}{s}},
 \end{equation*}
 a contradiction to \eqref{psss}. Therefore, $S=0.$ As a consequence of \eqref{eq328} $$\|u_j-u_\infty\|_{X_0^{s,2}(\Omega)}=0$$ as $j \rightarrow \infty.$ This particularly shows that $(u_j)_{j}$ has a convergent subsequence converging to $u_\infty$ strongly. Hence, $\mathfrak{I}_\lambda$ satisfies the Palais-Smale condition at any level $c$ satisfying \eqref{psss}, completing the proof of this lemma. 
\end{proof}

Next, we establish the following result to verify the condition $(c)$ of Theorem \ref{critithm}. 
\begin{lemma} \label{lem4.2}
    The functional $\mathfrak{I}_\lambda$ satisfies the assumption $(c)$ of Theorem \ref{critithm}.
\end{lemma}

\begin{proof}
    Assume $\lambda^*$ satisfies \eqref{eqq1.11} as in the statement of Theorem \ref{mainthm}. Then
    $$\lambda^*=\lambda_k\,\quad\text{for some}\,\, k \in \mathbb{N}.$$
    By hypothesis, $m$ is the multiplicity of $\lambda^*.$ Thus, we have 
    $$\lambda^*=\lambda_1=\lambda_2 \quad k=1$$
    \begin{equation} \label{eqq4.32}
        \lambda_{k-1}<\lambda^*=\lambda_k=\cdots = \lambda_{k+m-1}<\lambda_{k+m} \quad \text{if} \,\,\, k \geq 2.
    \end{equation}
    Let us first note that, under condition \eqref{eqq1.12}, the parameter $\lambda$ is such that 
    \begin{equation} \label{eqq4.33}
        \lambda>0.
    \end{equation}
    To see this, using the definition of $\lambda^*$ and taking into account Theorem \ref{ev-mainthm}(v) , we obtain 
    \begin{equation} \label{eqq4.34}
        \lambda^* \geq \lambda_1.
    \end{equation}
    On the other hand, the characterization of the first eigenvalue of $\lambda_1$ gives that (see \eqref{eigva})
    \begin{equation} \label{eigva1}
     \lambda_1:=\inf_{u \in X_0^{s,2}(\Omega) \backslash \{0\}} \frac{ \int_{\G \times \G}  \frac{|u(x)-u(y)|^2}{|y^{-1}x|^{Q+2s}}dxdy}{\int_{\Omega} |u(x)|^2\,dx}. 
 \end{equation}
 Now, making use of H\"older inequality 
 we have 
$$\int_{\Omega} |u(x)|^2 dx \leq |\Omega|^{2s/Q} \left( \int_{\Omega} |u(x)|^{2_s^*} dx\right)^{2/2^*_s}$$ and therefore, \eqref{eigva1} gives
$\lambda_1 \geq C_{2_s^*, Q} |\Omega|^{-2s/Q}.$
Thus \eqref{eqq4.34} yields that 
\begin{equation} \label{eqqex}
    \lambda^* \geq C_{2_s^*, Q} |\Omega|^{-2s/Q}.
\end{equation}
Therefore, \eqref{eqqex} along with condition \eqref{eqq1.12}, gives the desired conclusion $\lambda>0.$

Utilizing the notation of abstract  abstract multiplicity  Theorem \ref{critithm}, we define
$$\mathcal{W}:=\text{span}\{e_1, e_2, \ldots, e_{k+m-1}\}$$
and 
$$\mathcal{V}:= \begin{cases}
    X_0^{s,2}(\Omega) & \text{if}\,\, k=1
   \\
   \{u \in X_0^{s,2}(\Omega): \langle u, e_j \rangle_{X_0^{s,2}(\Omega) }=0 \quad \forall j=1,2,\ldots, k-1 \} & \text{if}\,\, k \geq 2.
\end{cases}$$

Note that $\mathcal{V}$ and $\mathcal{W}$ are closed subset of $X_0^{s,2}(\Omega) $ and 
\begin{equation}
    \text{dim} \mathcal{W}=k+m-1\quad\quad 
    \text{codim} \mathcal{V}=k-1,  
\end{equation} so that Theorem \ref{critithm}(c)(iii) satisfied. Now, we will show that $\mathfrak{I}_\lambda$ possess the geometric conditions (c)(i)-(ii) of Theorem \ref{critithm}. For this purpose, suppose that $u \in \mathcal{W}.$ Then 
$$u(x)= \sum_{i=1}^{k+m-1} u_i e_i(x),$$
where $u_i \in \mathbb{R}$ and $i=1,2, \ldots, k+m-1.$ By Theorem \ref{ev-mainthm} (vii),  $(e_k)_{k \in \N}$  is an orthonormal basis of $L^2(\Omega)$ and an orthogonal basis of $X_0^{s,2}(\Omega).$  Thus, taking \eqref{eqq4.32} into account, we obtain 
$$\|u\|_{X_0^{s,2}(\Omega)}^2= \sum_{i=1}^{k+m-1} u_i^2 \|e_i\|_{X_0^{s,2}(\Omega)}^2=  \sum_{i=1}^{k+m-1} u_i^2 \lambda_i \leq \lambda_k \sum_{i=1}^{k+m-1} u_i^2 \lambda_k \|u\|_{L^2(\Omega)}^2=\lambda^*  \|u\|_{L^2(\Omega)}^2 $$
 and consequently using H\"older inequality we get 
 \begin{align} \label{eq4.38}
       \mathfrak{I}_{\lambda}(u) &:= \frac{1}{2} \int_{\G \times \G} \frac{|u(x)-u(y)|^2}{|y^{-1}x|^{Q+2s}} dx dy - \frac{1}{2_s^*} \int_{\Omega} |u(x)|^{2_s^*} dx - \frac{\lambda}{2} \int_{\Omega} |u(x)|^2\, dx 
   \nonumber    \\&\leq \frac{1}{2} (\lambda^*-\lambda) \int_{\Omega} |u(x)|^2\,dx - \frac{1}{2_s^*} \int_{\Omega} |u(x)|^{2_s^*} dx \nonumber \\&\leq  \frac{1}{2} (\lambda^*-\lambda) |\Omega|^{\frac{2s}{Q}} \left( \int_{\Omega} |u(x)|^{2_s^*}\,dx \right)^{\frac{2}{2^*}} - \frac{1}{2_s^*} \int_{\Omega} |u(x)|^{2_s^*} dx. 
 \end{align}

Next, for $t\geq 0$, we define a function

$$h(t)= \frac{1}{2}(\lambda^*-\lambda) |\Omega|^{\frac{2s}{Q}} t-t^{2_s^*-1}.$$ It is easy to note that $g'(t) \geq 0$ if and only if 
$$t \leq \tilde{t}= \left[(\lambda^*-\lambda)|\Omega|^{\frac{2s}{Q}} \right]^{1/(2_s^*-2)}.$$
Therefore, $h$ attained maxima at $t=\tilde{t}$ and so, for any $t\geq 0$
\begin{equation} \label{eq4.39}
    h(t) \leq \max_{t\geq 0}h(t) =h(\tilde{t})= \frac{s}{Q} (\lambda^*-\lambda)^{\frac{Q}{2s}}|\Omega|.
\end{equation}

Thus, from \eqref{eq4.38} and \eqref{eq4.39} we get 
\begin{equation}
    \sup_{u \in \mathcal{W}} \mathfrak{I}_\lambda(u)= \max_{t\geq 0}h(t) = \frac{s}{Q} (\lambda^*-\lambda)^{\frac{Q}{2s}}|\Omega|.
\end{equation}
Note that as a result of condition \eqref{eqq1.12}, we have
$$ \frac{s}{Q} (\lambda^*-\lambda)^{\frac{Q}{2s}}|\Omega|>0.$$
Therefore, condition (c)(i) of Theorem \ref{critithm} is satisfied with $\beta'= \frac{s}{Q} (\lambda^*-\lambda)^{\frac{Q}{2s}}|\Omega|.$ Now, let us prove that $\mathfrak{I}_\lambda$ satisfies condition (c)(ii).  For every $u \in \mathcal{V},$ the inequality 
$$\|u\|^2_{X_0^{s,2}(\Omega)} \geq \lambda^* \|u\|^2_{L^2(\Omega)} $$ holds. To see this, first note that, if $u=0$ then the inequality trivially holds. Now, for $u \in \mathcal{V} \backslash \{0\}$ it follows from the variational characterization of $\lambda^*=\lambda_k$ given by 
$$\lambda_k= \inf_{u \in \mathcal{V} \backslash \{0\}} \frac{ \int_{\G \times \G}  \frac{|u(x)-u(y)|^2}{|y^{-1}x|^{Q+2s}}dxdy}{\int_{\Omega} |u(x)|^2\,dx}$$
as discussed in Theorem \ref{ev-mainthm}. Thus, by the definition of $C_{2^*_s, \Omega}$ and taking into account that $\lambda>0$ (see \eqref{eqq4.33}), it follows that 
\begin{align}
    \mathfrak{I}_\lambda(u) &\geq \frac{1}{2} \left(1-\frac{\lambda}{\lambda^*} \right) \|u\|_{X_0^{s,2}(\Omega)}^2- \frac{1}{2_s^* C_{2_s^*, \Omega}^{\frac{1}{2}}} \|u\|^{2_s^*}_{X_0^{s,2}(\Omega)} \nonumber \\& = \|u\|_{X_0^{s,2}(\Omega)}^2 \left(\frac{1}{2} \left(1-\frac{\lambda}{\lambda^*} \right) - \frac{1}{2_s^* C_{2_s^*, \Omega}^{\frac{1}{2}}} \|u\|^{2_s^*-2} _{X_0^{s,2}(\Omega)}\right).
\end{align}
Now, let $u \in \mathcal{V}$ such that $\|u\|_{X_0^{s,2}(\Omega)}=\rho>0.$ Since $2_s^*>2$ we can choose $\rho$ sufficiently small, say $\rho \leq \tilde{\rho}$ with $\tilde{\rho}$ so that 
\begin{equation} \label{eq4.42}
    \frac{1}{2} \left(1-\frac{\lambda}{\lambda^*} \right) - \frac{1}{2_s^* C_{2_s^*, \Omega}^{\frac{1}{2}}} \|u\|^{2_s^*-2} _{X_0^{s,2}(\Omega)}>0
\end{equation}
 and 
\begin{align}\label{eq4.43}
    \rho^2 \left(\frac{1}{2} \left(1-\frac{\lambda}{\lambda^*} \right) - \frac{1}{2_s^* C_{2_s^*, \Omega}^{\frac{1}{2}}} \rho^{2_s^*-2} \right) <\rho^2 \left(\frac{1}{2} \left(1-\frac{\lambda}{\lambda^*} \right) \right)<\beta':=\frac{1}{Q} (\lambda^*-\lambda)^{Q/2}|\Omega|,
\end{align}
 concluding the proof of this lemma.
\end{proof}

Now, we present the proof of Theorem \ref{mainthm}.
\begin{proof}[Proof of Theorem \ref{mainthm}]
    As we have discussed already, $\mathfrak{I}_\lambda$ satisfies assumption $(a)$ of Theorem \ref{critithm}. Next, Lemma \ref{lemm4.1} make sure that $\mathfrak{I}_\lambda$ fulfill condition $(b)$ of Theorem \eqref{critithm} with $$\beta=\frac{s}{Q} C_{2_s^*, \Omega}^{Q/s}.$$
Further, by Lemma \ref{lem4.2} we get that $\mathfrak{I}_\lambda$ verifies condition $(c)$ such that 
$$\beta'= \frac{s}{Q} (\lambda^*-\lambda)^{\frac{Q}{2s}}|\Omega|~\text{and}~\rho=\tilde{\rho}$$
with
$$\delta= \rho^2 \left(\frac{1}{2} \left(1-\frac{\lambda}{\lambda^*} \right) - \frac{1}{2_s^* C_{2_s^*, \Omega}^{\frac{1}{2}}} \rho^{2_s^*-2} \right).$$
Note that $0<\delta<\beta'<\beta$, thanks to  conditions \eqref{eq4.42}, \eqref{eq4.43} and assumption \eqref{eqq1.12}. Therefore, as an application Theorem \ref{critithm} , the functional $\mathfrak{I}_\lambda$ has $m$ pairs $\{-u_{\lambda, i}, u_{\lambda, i}\}$ of critical points whose critical value $\mathfrak{I}_\lambda(\pm u_{\lambda, i})$ are such that 
\begin{equation}\label{eqq4.44}
    0<\delta \leq \mathfrak{I}_{\lambda}(\pm u_{\lambda, i}) \leq \beta'~\text{for any}~i=1,2,\ldots, m.
\end{equation}
Since $\mathfrak{I}_\lambda(0)=0$ and \eqref{eqq4.44} holds true, we note that these critical points are all different from the trivial function. This shows that the problem \eqref{pro1intro} admits $m$ pairs of nontrivial weak solutions.

Next, let us fix $i \in \{1, 2, \ldots, m\}.$ By \eqref{eqq4.44} we get 
\begin{align*}
    \beta' \geq \mathfrak{I}_\lambda(u_{\lambda, i})= \mathfrak{I}_\lambda(u_{\lambda, i}) - \frac{1}{2} \langle \mathfrak{I}_\lambda'(u_{\lambda, i}), u_{\lambda, i} \rangle= \left( \frac{1}{2}-\frac{1}{2_s^*} \right) \|u_{\lambda, i}\|_{L^{2_s^*}(\Omega)}^{2_s^*}= \frac{s}{Q} \|u_{\lambda, i}\|_{L^{2_s^*}(\Omega)}^{2_s^*}.
\end{align*}
Thus, passing to the limit as $\lambda \rightarrow \lambda^*,$ it follows that 
\begin{equation} \label{46}
    \|u_{\lambda, i}\|_{L^{2_s^*}(\Omega)}^{2_s^*} \rightarrow 0
\end{equation}
by noting $\beta' \rightarrow 0$ as $\lambda \rightarrow \lambda^*.$
    Then, using continuous embedding $L^{2_s^*} \hookrightarrow  L^2(\Omega)$  (since $\Omega$ is bounded), we immediately get
    \begin{equation}\label{47}
        \|u_{\lambda, i}\|_{L^{2}(\Omega)}^{2} \rightarrow 0\quad\quad \text{as}\,\,\,\lambda \rightarrow \lambda^*.
    \end{equation}
    Therefore, arguing as above, we get 
    $$\beta' \geq \mathfrak{I}_\lambda(u_{\lambda, i})= \frac{1}{2} \|u\|^2_{X_0^{s,2} (\Omega)}-\frac{\lambda}{2} \|u_{\lambda, i}\|^2_{L^2(\Omega)}-\frac{1}{2_s^*} \|u_{\lambda, i}\|^{2_s^*}_{L^{2_s^*}(\Omega)}$$
    which combined with \eqref{46} and \eqref{47} gives
    $$\|u\|_{X_0^{s,2} (\Omega)} \rightarrow 0\quad\text{as}\quad \lambda \rightarrow \lambda^*.$$
    This concludes the proof of Theorem \ref{mainthm}.
\end{proof}

\section*{Conflict of interest statement}
On behalf of all authors, the corresponding author states that there is no conflict of interest.

\section*{Data availability statement}
Data sharing is not applicable to this article as no datasets were generated or analysed during the current study.

\section*{Acknowledgement}
SG acknowledges the research facilities available at the Department of Mathematics, NIT Calicut under the DST-FIST support, Govt. of India [Project no. SR/FST/MS-1/2019/40 Dated. 07.01 2020.].   VK is supported by the FWO Odysseus 1 grant G.0H94.18N: Analysis and Partial Differential Equations, the Methusalem programme of the Ghent University Special Research Fund (BOF) (Grant number 01M01021) and by FWO Senior Research Grant G011522N.


\end{document}